\documentclass{amsart}

\usepackage{amsfonts, amsmath, amscd}
\usepackage[psamsfonts]{amssymb}
\usepackage{amssymb}
\usepackage[usenames]{color}
\usepackage{hyperref}
\usepackage[all]{xy}
\numberwithin{equation}{section}

\newtheorem{theorem}{Theorem}[section]
\newtheorem{corollary}[theorem]{Corollary}
\newtheorem{proposition}[theorem]{Proposition}
\newtheorem{lemma}[theorem]{Lemma}

\theoremstyle{definition}
\newtheorem{example}[theorem]{Example}

\newtheorem{problem}[theorem]{Problem}
\newtheorem{remark}[theorem]{Remark}
\theoremstyle{plain}

\newcommand{\w}{\omega}
\newcommand{\K}{\mathcal K}

\newcommand{\Ra}{\Rightarrow}

\newcommand{\IR}{\mathbb R}
\newcommand{\IN}{\mathbb N}

\newcommand{\V}{\mathcal V}

\newcommand{\U}{\mathcal U}

\newcommand{\E}{\mathcal E}
\newcommand{\e}{\varepsilon}
\newcommand{\Lc}{\mathsf{L}}
\newcommand{\Lin}{\mathsf{V}}
\newcommand{\FG}{\mathsf{F}}
\newcommand{\AG}{\mathsf{A}}
\newcommand{\BG}{\mathsf{B}}

\newcommand{\St}{\mathcal{S}t}
\newcommand{\supp}{\mathrm{supp}}

\newcommand{\cbox}{\boxdot}

\newcommand{\Tau}{\mathcal{T}}
\newcommand{\conv}{\mathrm{conv}}

\begin{document}

\title[$\mathfrak G$-bases in free (locally convex) topological vector spaces]{$\mathfrak G$-bases in free (locally convex) topological vector spaces}
\author{Taras Banakh and Arkady Leiderman}
\address{T. Banakh: Department of Mathematics, Ivan Franko National University of Lviv (Ukraine) and
 Instytut Matematyki, Jan Kochanowski University in Kielce (Poland)}
\email{t.o.banakh@gmail.com}
\address{A. Leiderman: Department of Mathematics, Ben-Gurion University of the Negev, Beer Sheva, P.O.B. 653, Israel}
\email{arkady@math.bgu.ac.il}

\date{\today}

\begin{abstract}
A topological space $X$ is defined to have a local {\em $\mathfrak G$-base}
if every point $x\in X$ has a neighborhood base $(U_\alpha)_{\alpha\in\w^\w}$  such that $U_\beta\subset U_\alpha$ for all
 $\alpha\le\beta$ in $\w^\w$.
 We prove that for every Tychonoff space $X$ the following conditions are equivalent:
\begin{enumerate}
\item the free locally convex space $\Lc(X)$ of $X$ has a local $\mathfrak G$-base;
\item the free topological vector space $\Lin(X)$ of $X$ has a local $\mathfrak G$-base;
\item the finest uniformity $\U(X)$ of $X$ admits a $\mathfrak G$-base and the function space $C(X)$ is $\w^\w$-dominated.
\end{enumerate}
The conditions (1)--(3) imply that every metrizable continuous image of $X$ is $\sigma$-compact and all finite powers of $X$ are countably tight.
If the space $X$ is separable, then the conditions (1)--(3) imply that $X$ is a countable union of compact metrizable spaces and (1)--(3) are equivalent to:
\begin{itemize}
\item[(4)]  the finest uniformity $\U(X)$ of $X$ has a $\mathfrak G$-base.
\end{itemize}
If the space $X$ is first-countable and perfectly normal, then the conditions (1)--(3) are equivalent to:
\begin{itemize}
\item[(5)] $X$ is metrizable and $\sigma$-compact.
\end{itemize}
If the space $X$ is countable, then the conditions (1)--(4) are equivalent to:
\begin{itemize}
\item[(6)] the space $X$ has a local $\mathfrak G$-base.
\end{itemize}
If $X$ is a $k$-space, then the conditions (1)--(3) are equivalent to:
\begin{itemize}
\item[(7)] the double function space $C_k(C_k(X))$ has a local $\mathfrak G$-base.
\end{itemize}
Under $\w_1<\mathfrak b$ the conditions (1)--(3) are equivalent to
\begin{itemize}
\item[(8)] the finest uniformity $\U(X)$ of $X$ is $\w$-narrow and has a $\mathfrak G$-base.
\end{itemize}
Under $\w_1=\mathfrak b$ the conditions (1)--(3) are not equivalent to (8).

\noindent These results resolve several open problems previously stated in the literature.
\end{abstract}
\maketitle

\section{Introduction and Main Results}
It is known that a uniform space is metrizable if and only if its uniformity  has a countable base (see \cite{Eng}, Theorem 8.1.21).  Natural compatible uniform structures exist in topological groups, and, in particular, in topological vector spaces. This fact implies
the classical metrization theorem of Birkhoff and Kakutani which states that a topological group is metrizable if and only if it admits a countable base of neighborhoods of the identity, or, in other words, a base indexed by $\omega$.

In this paper we consider uniform spaces and topological spaces which have neighborhood bases indexed by a more complicated directed set $\w^\w$, which consists of all functions from $\w$ to $\w$ and is endowed with the natural partial order (defined by $f\le g$ iff $f(n)\le g(n)$ for all $n\in\w$).
Cardinal invariants of topological spaces, associated with combinatorial properties of the poset $\w^\w$
have been investigated for many years (see \cite{Douwen}, \cite{Vaug}).

A topological space $X$ is defined to have a {\em $\mathfrak G$-base} at a point $x\in X$ if there exists
a neighborhood base $(U_\alpha)_{\alpha\in\w^\w}$ at $x$ such that $U_\beta\subset U_\alpha$ for all
elements $\alpha\le\beta$ in $\w^\w$. We shall say that a topological space has a {\em local $\mathfrak G$-base} if it has a $\mathfrak G$-base at each point $x\in X$. Evidently, a topological group $G$ has a {\em local $\mathfrak G$-base} if it has a $\mathfrak G$-base at the identity $e\in G$.

For a topological space $X$ by $\U(X)$ we denote the finest uniformity on $X$, compatible with the topology of $X$. It is generated by the base consisting of entourages $[\rho]_{<1}:=\{(x,y)\in X\times X:\rho(x,y)<1\}$ where $\rho$ runs over all continuous pseudometrics on $X$. A uniformity $\U$ on a space $X$ is defined
\begin{itemize}
\item to be {\em $\w$-narrow} if for every entourage $U\in\U$ there exists a countable set $C\subset X$ such that $X=\{x\in X:\exists c\in C$ with $(x,c)\in U\}$;
\item  to have a {\em $\mathfrak G$-base} if there is a base of entourages $\{U_\alpha\}_{\alpha\in\w^\w}\subset\U(X)$  such that $U_\beta\subset U_\alpha$ for all $\alpha\le\beta$ in $\w^\w$.
\end{itemize}

Cascales and Orihuela were the first who considered uniform spaces admitting a $\mathfrak G$-base \cite{CO}.
They proved that compact spaces with this property are metrizable (recall that a compact space carries a unique compatible uniformity).
For the first time the concept of a $\mathfrak{G}$-base  appeared in \cite{FKLS} as a tool for studying locally convex spaces that belong to the class $\mathfrak{G}$ introduced by Cascales and Orihuela \cite{CO}. A systematic study of locally convex spaces and topological groups with a $\mathfrak{G}$-base has been started in \cite{GabKakLei_1}, \cite{GabKakLei_2} and continued in \cite{GabKak_1},
 \cite{GabKak_2}, \cite{LRZ}. The notion of a $\mathfrak{G}$-base in groups implicitly appears also in \cite{CFHT}.


Our current paper is devoted to detecting $\mathfrak G$-bases in two important classes of topological vector spaces: free locally convex spaces $\Lc(X)$ and free topological vector spaces $\Lin(X)$ over topological (or uniform) spaces $X$.

For a topological space $X$ its {\em free locally convex space} is a pair $(\Lc(X),\delta_X)$ consisting of a locally convex topological vector space $\Lc(X)$ and a continuous map $\delta_X:X\to \Lc(X)$ such that for every continuous map $f:X\to Y$ into a locally convex topological vector space $Y$ there exists a continuous linear operator $\bar f:\Lc(X)\to Y$ such that $\bar f\circ\delta_X=f$. Deleting the phrase ``locally convex'' in this definition, we obtain the definition of a {\em free topological vector space} $(\Lin(X),\delta_X)$ of $X$. All  topological vector spaces considered in this paper are over the field $\IR$ of real numbers; also
all topological vector spaces are assumed to be Hausdorff, and therefore Tychonoff.

It is well-known that for every topological space $X$ its free
(locally convex) topological vector space exists and is unique up to a topological isomorphism.

It is worth  mentioning that in the case of a Tychonoff topological space $X$,
the canonical map $\delta_X$ is a closed topological embedding, so we can identify the space $X$ with its image $\delta_X(X)$ and say that $X$ algebraically generates
 the free topological objects $\Lc(X)$ and $\Lin(X)$.

For an infinite Tychonoff space $X$ the free objects $\Lc(X)$ and $\Lin(X)$ are not first-countable (so do not admit neighborhood bases indexed by $\w$). On the other hand, it was  shown in \cite{GabKakLei_2} that the space $\Lc(X)$ admits a  local $\mathfrak{G}$-base provided $X$ is compact and metrizable (more generally, $X$ is a submetrizable $k_\w$-space). This result will be substantially generalized in our Theorem \ref{t:limit} and also extended to the spaces $\Lin(X)$.


Results of our research will be presented in two separate papers.
The current paper concentrates on detecting local $\mathfrak G$-bases in
free locally convex space $\Lc(X)$ and
free topological vector space $\Lin(X)$ over topological or uniform spaces $X$.
In the subsequent  paper \cite {BL2} we shall detect local $\mathfrak G$-bases in free topological groups $\FG(X)$, free Abelian topological groups $\AG(X)$ and  free Boolean topological groups $\BG(X)$ over topological or uniform spaces $X$.

Let us mention that for any Tychonoff space $X$ the free Abelian topological group $\AG(X)$ naturally embeds into the
free locally convex space $\Lc(X)$ \cite{Tkach, Usp}, therefore if $\Lc(X)$ has a local $\mathfrak G$-base, then $\AG(X)$ has a local $\mathfrak G$-base, too. On the contrary,
a recent relevant result obtained in \cite {LPT} shows that  $\Lc(X)$ need not have a local $\mathfrak G$-base even if $\AG(X)$ does. By \cite{LPT}, the free Abelian topological group $\AG(X)$ of a Tychonoff space $X$ has a local $\mathfrak G$-base if and only if the finest  uniformity $\U(X)$ of $X$ has a $\mathfrak G$-base. On the other hand, we shall prove that the free locally convex space $\Lc(X)$ of a Tychonoff space $X$ has a local $\mathfrak G$-base if and only if the finest uniformity $\U(X)$ of $X$ has a $\mathfrak G$-base and the space $C(X)$ of continuous real-valued functions on $X$ is {\em $\w^\w$-dominated} in the sense that it contains a subfamily $\{f_\alpha\}_{\alpha\in\w^\w}\subset C(X)$ such that $f_\alpha\le f_\beta$ for all $\alpha\le\beta$ in $\w^\w$, and for every $f\in C(X)$ there exists $\alpha\in\w^\w$ such  that $f\le f_\alpha$.




One of the main results of this paper is Theorem \ref{t:Lin+Lc} which is described below in a simplified form.
This theorem resolves completely an open problem posed in \cite{GabKakLei_2} (Question 4.15): {\em for which spaces $X$ the free locally convex space
$\Lc(X)$  has a local $\mathfrak G$-base?}

Namely, in Theorem~\ref{t:Lin+Lc} we shall prove that {for a Tychonoff space $X$
 the following conditions are equivalent:
\begin{enumerate}
\item[$(\Lc)$] the free locally convex space $\Lc(X)$ of $X$ has a local $\mathfrak G$-base;
\item[$(\Lin)$] the free topological vector space $\Lin(X)$ of $X$ has a local $\mathfrak G$-base;
\item[$(\U \mathsf C)$] the finest uniformity $\U(X)$ of $X$ has a $\mathfrak G$-base  and the function space $C(X)$ is $\w^\w$-dominated.
\end{enumerate}
The conditions $(\Lc),(\Lin)$ imply that every metrizable continuous image of $X$ is $\sigma$-compact, and all finite powers of $X$ are countably tight.
If the space $X$ is separable or $\w_1<\mathfrak b$, then the conditions  $(\Lc),(\Lin)$  imply that $X$ is a countable union of compact metrizable spaces.

\noindent If the space $X$ is separable, then  $(\Lc),(\Lin)$ are equivalent to
\begin{itemize}
\item[$(\U)$]  the finest uniformity $\U(X)$ of $X$ has a $\mathfrak G$-base.
\end{itemize}
If the space $X$ is first-countable and perfectly normal, then the conditions  $(\Lc),(\Lin)$  are equivalent to
\begin{itemize}
\item[$(M\sigma)$] $X$ is metrizable and $\sigma$-compact.
\end{itemize}
If the space $X$ is countable, then the conditions  $(\Lc),(\Lin)$  are equivalent to
\begin{itemize}
\item[$(\mathfrak G)$] the space $X$ has a local $\mathfrak G$-base.
\end{itemize}
If $X$ is a $k$-space, then  the conditions  $(\Lc),(\Lin)$  are equivalent to
\begin{itemize}
\item[$(C_k^2)$] the double function space $C_k(C_k(X))$ has a local $\mathfrak G$-base.
\end{itemize}
\smallskip

Here for a Tychonoff space $X$ by $C_k(X)$ we denote the space of real-valued continuous functions on $X$, endowed with the compact-open topology. The equivalence $(\Lc)\Leftrightarrow (C_k^2)$ (holding for $k$-spaces) answers Question 4.18 posed in \cite{GabKakLei_2}.
\smallskip

Concerning the condition ($\U$) let us mention the relevant characterization \cite{Ginsburg} of Tychonoff spaces $X$ whose finest uniformity $\U(X)$ has a countable base: those are exactly  metrizable spaces with compact set of non-isolated points.
\smallskip

In Theorem~\ref{t:Lin+Lc} we shall also obtain an interesting consistency result depending on the relation between the cardinals $\w_1$ and $\mathfrak b$. Following \cite{Douwen} by $\mathfrak b$ we denote the smallest cardinality of a subset $B\subset \w^\w$ which is not dominated by a countable set in $\w^\w$. It is well-known that $\w_1\le\mathfrak b$ and both  $\w_1=\mathfrak b$ and $\w_1<\mathfrak b$ are consistent with ZFC, see \cite{Douwen} or \cite{Vaug}. In particular, $\w_1<\mathfrak b$ holds under MA$+\neg$CH.

Namely, in Theorem~\ref{t:Lin+Lc}(8,9) we shall prove that under $\w_1<\mathfrak b$ for a Tychonoff space $X$ the conditions $(\Lc),(\Lin)$ are equivalent to
\begin{itemize}
\item[$(\w\U)$] the finest uniformity $\U(X)$ of $X$ is $\w$-narrow and has a $\mathfrak G$-base.
\end{itemize}}
On the other hand, under $\w_1=\mathfrak b$, the conditions $(\Lc),(\Lin)$ are not equivalent to $(\w\U)$. This means that the equivalence of the conditions $(\Lc,\Lin)$ and $(\w\U)$ is independent of ZFC.
\smallskip

The language of uniform spaces is  systematically used throughout the paper.
First, in Sections~\ref{s:Lc}, \ref{s:Lin} we characterize those uniform spaces $X$ whose free (locally convex) topological vector  spaces admit a $\mathfrak G$-base, and then (in Section~\ref{s:Lin+Lc}) we apply the obtained ``uniform'' results the topological case. In Sections~\ref{s:Set} and \ref{s:U} we study uniform spaces whose uniformity has a $\mathfrak G$-base. The main result here is the metrization Theorem~\ref{t:metr} which says in particular that a first-countable perfectly normal space $X$ is metrizable if and only if the topology of $X$ is generated by a uniformity possessing a $\mathfrak G$-base. In Section~\ref{s:CX} we investigate Tychonoff spaces $X$ whose function space $C(X)$ is $\w^\w$-dominated.

\section{Reductions between posets}\label{s:pre}
For handling $\mathfrak G$-bases, it is convenient to use the language of reducibility of posets. By a {\em poset} we understand a partially ordered set, i.e., a set endowed with a partial order.

A function $f:P\to Q$ between two posets is called
\begin{itemize}
\item  {\em monotone} if $f(p)\le f(p')$ for all $p\le p'$ in $P$;
\item {\em cofinal} if for each $q\in Q$ there exists $p\in P$ such that $q\le f(p)$.
\end{itemize}

Given two partially ordered sets $P,Q$ we shall write $P\succcurlyeq Q$ or $Q\preccurlyeq P$ and say that $Q$ {\em reduces} to $P$ if there exists a monotone cofinal map $f:P\to Q$. Also we write $P\cong Q$ if $P\preccurlyeq Q$ and $P\succcurlyeq Q$. This kind of reducibility of posets  is a bit stronger than the Tukey reducibility \cite{DT}, which requires the existence of a cofinal (but not necessarily monotone) map $f:P\to Q$.

We shall also use a related notion of $P$-dominance. Given two posets $P,Q$, we shall say that a subset $D\subset Q$ is {\em $P$-dominated in $Q$} if there exists a monotone map $f:P\to Q$ such that for every $x\in D$ there exists $y\in P$ with $x\le f(y)$. It follows that a poset $Q$ reduces to $P$ if and only if $Q$ is $P$-dominated in $Q$. In this case we shall also say that the poset $Q$ is {\em $P$-dominated}.

For a point $x$ of a topological space $X$ by $\Tau_x(X)$ we denote the poset of all neighborhoods of $x$ in $X$, endowed with the partial order of converse inclusion $U\le V$ iff $V\subset U$.
Observe that a topological space $X$ has a $\mathfrak G$-base at a point $x\in X$ if and only if  $\w^\w\succcurlyeq\Tau_x(X)$ where $\w^\w$ is the poset of all functions from $\w$ to $\w$, endowed with the partial order $\le$ defined by $f\le g$ iff $f(n)\le g(n)$ for all $n\in\w$.

In the sequel we equip every topological vector space $L$ with the natural uniformity generated by the base consisting of entourages $\{(x,y)\in L:x-y\in U\}$ where $U\in\Tau_0(L)$. Here  $\Tau_0(L)$ denotes the poset of all neighborhoods of zero in $L$.

For a set $X$ by $\IR^X$ we denote the set of all real-valued functions on $X$ endowed with the pointwise partial order $\le$ defined by $f\le g$ iff $f(x)\le g(x)$ for all $x\in X$. If $X$ is a topological space, then $C(X)$ stands for the poset of all continuous real-valued functions on $X$ endowed with the partial order inherited from $\IR^X$.

Given a uniform space $X$ by $\U(X)$ we denote the uniformity of $X$ endowed with the partial order of converse inclusion $U\le V$ iff $V\subset U$. For an entourage $U\in\U(X)$ and a point $x\in X$ by $U[x]$ we denote the $U$-ball $\{y\in X:(x,y)\in U\}\subset X$ centered at $x$. Each uniform space $X$ will be endowed with the (Tychonoff) topology consisting of all subsets $W\subset X$ such that for every $x\in W$ there exists an entourage $U\in\U(X)$ such that $U[x]\subset W$.

Conversely, each Tychonoff space $X$ will be equipped with the {\em  finest uniformity} $\U(X)$ generated by the base consisting of entourages $[d]_{<1}=\{(x,y):d(x,y)<1\}$ where $d$ runs over all continuous pseudometrics on $X$.

By $C_u(X)\subset\IR^X$ we denote the space of all uniformly continuous real-valued functions on the uniform space $X$. Observe that for a topological space $X$ equipped with the finest uniformity $\U(X)$  the function space $C_u(X)$ coincides with $C(X)$.

A function $f:X\to Y$ between uniform spaces is called {\em $\w$-continuous} if for any entourage $U\in\U(Y)$ there exists a countable family $\V\subset\U(X)$ such that for every $x\in X$ there exists $V\in\V$ satisfying $f(V[x])\subset U[f(x)]$.  It is clear that each $\w$-continuous function $f:X\to Y$ between uniform space is continuous with respect to the topologies generated by the uniformities.

By $C_\w(X)$ we denote the poset  of all $\w$-continuous real-valued functions on a uniform space. It follows that  $C_u(X)\subset C_\w(X)\subset C(X)\subset\IR^X$.

For a topological space $X$ let $\K(X)$ be the poset of compact subsets of $X$, endowed with the inclusion partial order $A\le B$ iff $A\subset B$.  The following fundamental fact was proved by Christensen \cite{Chris} (see also \cite[6.1]{kak}).

\begin{theorem}[Christensen]\label{t:Chris} A metrizable space $X$ is Polish iff $\w^\w\succcurlyeq \K(X)$.
\end{theorem}

We shall apply this Christensen's theorem to prove the following result generalizing Arhangel'skii--Calbrix Theorem \cite{AC} (see also \cite[T9.9]{kak}). A topological space $X$ is called {\em $\sigma$-compact} if $X$ is a countable union of its compact subspaces.

\begin{theorem}\label{t:dominat} If for a uniform space $X$ the set $C_\w(X)$ is $\w^\w$-dominated in $\IR^X$, then for every $\w$-continuous map $f:X\to M$ to a metric space $M$ the image $f(X)$ is $\sigma$-compact.
\end{theorem}

\begin{proof} We lose no generality assuming that $M = f(X)$. First we consider the case of a totally bounded metric space $M$. In this case the completion $\bar M$ of $M$ is compact. Denote by $d$ the metric of the compact metric space $\bar M$. Let $\{\varphi_\alpha\}_{\alpha\in\w^\w}\subset \IR^X$ be a subset witnessing that the set $C_\w(X)$ is $\w^\w$-dominated in $\IR^X$.   Replacing each function $\varphi_\alpha$ by $\max\{1,\varphi_\alpha\}$, we can assume that $\varphi_\alpha(X)\subset[1,\infty)$.

For every $\alpha\in\w^\w$ consider the open set $U_\alpha=\bigcup_{x\in X}B(f(x),1/\varphi_\alpha(x))$ and the compact set
 $K_\alpha=\bar M\setminus U_\alpha$ in $\bar M$. Here by $B(x,\e)=\{y\in\bar M:d(x,y)<\e\}$ we denote the open $\e$-ball centered at a point $x$ of the metric space $\bar M$. Observe that for any $\alpha\le\beta$ in $\w^\w$ the inequality $\varphi_\alpha\le \varphi_\beta$ implies the inclusion $K_\alpha\subset K_\beta$. We claim that the family $(K_\alpha)_{\alpha\in\w^\w}$ is cofinal in $\K(\bar M\setminus M)$. Given any compact set $K\subset\bar M\setminus M$, consider the $\w$-continuous function $\varphi:X\to\IR$ defined by
 $\varphi(x)=1/d(f(x),K)$ where $d(f(x),K)=\min_{y\in K}d(f(x),y)$. Find $\alpha\in\w^\w$ such that $\varphi_\alpha\ge\varphi$ and observe that $K\subset K_\alpha$, witnessing that $\w^\w\succcurlyeq \K(\bar M\setminus M)$. By Theorem~\ref{t:Chris},
 the space $\bar M\setminus M$ is Polish and hence of type $G_\delta$ in $\bar M$. Then $M$ is $\sigma$-compact, being of type $F_\sigma$ in the compact space $\bar M$.

Next, we consider the case of separable metric space $M = f(X)$. In this case we can take any homeomorphism $h:M\to T$ onto a totally bounded metric space $T$ and observe that the $\w$-continuity of the map $f:X\to M$ implies the $\w$-continuity of the composition $h\circ f:X\to T$. By the preceding case, the image $h\circ f(X)$ is $\sigma$-compact. Since $h$ is a homeomorphism, the space $f(X) = M$ is
 $\sigma$-compact too.

Finally, we can prove the general case of arbitrary metric space $M$.
By the preceding case, it is sufficient to check that the image $f(X) = M$ is separable. Assuming that $M = f(X)$ is non-separable, we can find a closed discrete subspace $D\subset M$ of cardinality $|D|=\w_1$. By \cite[Corollary 1]{BP}, the separable Hilbert space $\ell_2$ contains an uncountable linearly independent compact set $K\subset\ell_2$. Fix a subset $D'\subset K$  of cardinality $|D'|=\w_1$ which is not $\sigma$-compact. Next, consider the linear hull $L$ of the set $D'$ in $\ell^2$ and observe that $D' = L\cap K$, so $D'$ is a closed subset of $L$. Take any surjective map $g:D\to D'$. By Dugundji Theorem \cite{Dug}, the map $g:D\to D'\subset L$ has a continuous extension $\bar g:M\to L$. Then the map $\bar g\circ f:X\to L$ is $\w$-continuous. Since the space $L$ is separable, the preceding case guarantees that the image $\bar g\circ f(X)$ is $\sigma$-compact and so is its closed subspace $D'=\bar g\circ f(X)\cap K$. But this contradicts the choice of $D'$.
\end{proof}

A uniform space $X$ is called {\em $\w$-narrow} if for every entourage $U\in\U(X)$ there exists a countable set $C\subset X$ such that $X=\bigcup_{x\in C}U[x]$.  It is well-known that a uniform space $X$ is $\omega$-narrow if and only if for every uniformly continuous map $f:X\to M$ to a metric space $M$ the image $f(X)$ is separable. This characterization and Theorem~\ref{t:dominat} imply:

\begin{corollary}\label{c:narrow} A uniform space $X$ is $\w$-narrow if the set $C_\w(X)$ is $\w^\w$-dominated in $\IR^X$.
\end{corollary}

\section{Some consistency results}\label{s:Set}

In this section we establish some properties of $\w$-narrow uniform spaces $X$ with a $\mathfrak G$-base, which hold under the set-theoretic assumption $\w_1<\mathfrak b$.

Following \cite{Tka}, we say that a subset $S$ of a topological space $X$ is {\em weakly separated} if every point $x\in S$ has a neighborhood $O_x\subset X$ such that for any distinct points $x,y\in S$ either $x\notin O_y$ or $y\notin O_x$.

A topological space $X$ is defined to be
\begin{itemize}
\item {\em cosmic} if $X$ has a countable network of the topology;
\item {\em weakly cosmic} if $X$ contains no uncountable weakly separated subspace.
\end{itemize}
By \cite{Tka}, each cosmic space is weakly cosmic and each weakly cosmic space is hereditarily separable and hereditarily Lindel\"of. By \cite[p.30]{Todo} under PFA (the Proper Forcing Axiom), a regular space is cosmic if and only if all its finite powers are weakly cosmic.

\begin{lemma}\label{l:ws} Assume that $\w_1<\mathfrak b$. If the topology of a Tychonoff space $X$ is generated by an $\w$-narrow uniformity $\U$ with a $\mathfrak G$-base, then $X$ is weakly cosmic.
\end{lemma}

\begin{proof}  Let $(U_\alpha)_{\alpha\in\w^\w}$ be a $\mathfrak G$-base of the uniformity $\U$. Assuming that $X$ is not weakly cosmic, we can find a weakly separated subset $S\subset X$ of cardinality $|S|=\w_1$. For every $x\in S$ choose a neighborhood $O_x\subset X$ such that for any distinct points $x,y\in S$ either $x\notin O_y$ or $y\notin O_x$.

For every $x\in S$ find an entourages $W_x,V_x\in\U$ such that $W_x[x]\subset O_x$ and $V_x^{-1}V_x\subset W_x$. For the entourage $V_x$ choose a function $\alpha(x)\in\w^\w$ such that $U_{\alpha(x)}\subset V_x$. The correspondence $x\mapsto\alpha(x)$ determines a function $\alpha:S\to\w^\w$. By definition of the cardinal $\mathfrak b>\w_1=|\alpha(S)|$, the set $\alpha(S)$ is dominated by some countable set $B\subset\w^\w$. By the Pigeonhole Principle, there exists a function $\beta\in B$ such that the set $S_\beta=\{x\in S:\alpha(x)\le \beta\}$ is uncountable.

 We claim that for any distinct points $x,y\in S_\beta$ we get $U_\beta[x]\cap U_\beta[y]=\emptyset$. Indeed, the choice of the neighborhoods $O_x,O_y$ guarantees that either $x\notin O_y$ or $y\notin O_x$. In the first case we observe that $U_\beta^{-1}U_\beta[y]\subset U_{\alpha(y)}^{-1}U_{\alpha(y)}[y]\subset V_y^{-1}V_y[y]\subset W_y[y]\subset O_y$ and hence $x\notin U_\beta^{-1}U_\beta[y]$, which implies $U_\beta[x]\cap U_\beta[y]=\emptyset$. In the second case the equality $U_\beta[x]\cap U_\beta[y]=\emptyset$ follows from the inclusions $U_\beta^{-1}U_\beta[x]\subset U_{\alpha(x)}^{-1}U_{\alpha(x)}[x]\subset V_x^{-1}V_x[x]\subset W_n[x]\subset O_x$. Now we see that the uniform space $X$ contains an uncountable disjoint family $\{U_\beta[x]:x\in S_\beta\}$ of $U_\beta$-balls, and hence $X$ cannot be $\w$-narrow.
This contradiction completes the proof.
\end{proof}

Lemma~\ref{l:ws} admits the following its improvement.

\begin{theorem}\label{t:ws} Assume that $\w_1<\mathfrak b$. If the topology of a Tychonoff space $X$ is generated by an $\w$-narrow uniformity $\U$ with a $\mathfrak G$-base, then the countable power $X^\w$ of $X$ is weakly cosmic.
\end{theorem}

\begin{proof} For every pair $P=(n,(U_n)_{n\in\w})\in\w\times\U^\w$ consider the entourage
$$V_P=\{((x_k)_{n\in\w},(y_k)_{n\in\w})\in X^\w\times X^\w:\forall k\le n\;\; (x_k,y_k)\in U_k\}\subset X^\w\times X^\w.$$
It is easy to see that the topology of the space $X^\w$ is generated by the uniformity $\V$,  generated by the base $\{V_P:P\in\w\times\U^\w\}$. The monotone cofinal map $\w\times\U^\w\to\V$, $P\mapsto V_P$, witnesses that $\w\times\U^\w\succcurlyeq \V$. Taking into account that the uniformity $\U$ has a $\mathfrak G$-base, we conclude that $\w^\w\cong\w\times\U^\w\succcurlyeq\V,$ which means that the uniformity $\V$ has a $\mathfrak G$-base.
Applying Lemma~\ref{l:ws}, we conclude that the space $X^\w$ is weakly cosmic.
\end{proof}

In the following corollary of Theorem~\ref{t:ws} we use PFA, the Proper Forcing Axiom. For a reader without set-theoretic inclinations it suffices to know that PFA is consistent with ZFC and has many nice implications (beyond Set Theory), see \cite{Baum} or \cite{Moore}.

\begin{corollary}\label{c:PFA} Assume PFA. If the topology of a Tychonoff space $X$ is generated by an $\w$-narrow uniformity $\U$ with a $\mathfrak G$-base, then the space $X$ is cosmic.
\end{corollary}

\begin{proof} It is well-known (see e.g. \cite[\S7]{Moore}) that PFA implies $\w_1<\mathfrak b=\w_2$. So, we can apply Theorem~\ref{t:ws} and conclude that $X^\w$ is weakly cosmic. By \cite[p.30]{Todo}, the space $X$ is cosmic (since all its finite powers are weakly cosmic).
\end{proof}

The following example shows that preceding results proved under $\w_1<\mathfrak b$ are not true in ZFC. We recall that a topological space $X$ is a {\em $P$-space} if each $G_\delta$-set in $X$ is open. A topological space $X$ has {\em countable extent} if $X$ contains no uncountable closed discrete subspaces.

\begin{example}\label{ex:PU} There exists a non-separable Tychonoff $P$-space $X$ with countable extent whose finest uniformity $\U(X)$ is $\w$-narrow and under $\w_1=\mathfrak b$ has a $\mathfrak G$-base.
\end{example}

\begin{proof}The space $X$ is the well-known example of the uncountable $\sigma$-product of doubletons, endowed with the $\w$-box topology. This space was introduced by Comfort and Ross in \cite{CR66} and also discussed in the book \cite[4.4.11]{AT}. Let us briefly remind the construction of $X$.

 In the uncountable power $2^{\w_1}$ of the doubleton $2=\{0,1\}$ consider the subset $X$ consisting of functions $x:\w_1\to 2$ with finite support $\supp(x)=x^{-1}(1)$. On the space $X$ consider the uniformity $\U$, generated by the base $\{U_\alpha\}_{\alpha\in\w_1}$ consisting of the entourages $$U_\alpha=\{(x,y)\in X\times X:x|\alpha=y|\alpha\}\mbox{ for $\alpha\in\w_1$}.$$ It is easy to see that the topology on $X$, generated by the uniformity $\U$,  turns $X$ into a non-separable $P$-space.

Next, we prove that the topological space $X$ has countable extent. To derive a contradiction, assume that  $X$ contains an uncountable closed discrete subset $D$.  By the $\Delta$-System Lemma \cite[p.~49]{Kunen}, there exist a finite set $\Delta\subset \w_1$ and an uncountable subset $D'\subset D$ such that $\supp(x)\cap\supp(y)=\Delta$ for all distinct elements $x,y\in D'$. We claim that the characteristic function $z\in X$ of the finite set $\Delta=\supp(z)$ is an accumulation point of the set $D'$. Given any neighborhood $O_z$ of $z$, find a countable ordinal $\alpha\in\w_1$ such that $U_\alpha[z]\subset O_z$. Since the family $(\supp(x)\setminus\Delta)_{x\in D'}$ is disjoint, the set $D''=\{x\in D':(\supp(x)\setminus \Delta)\cap [0,\alpha)\ne\emptyset\}$ is countable. Observe that for every $x\in D'\setminus D''$ we get $\supp(x)\cap [0,\alpha)=\Delta$ and hence $x|\alpha=z|\alpha$ and $x\in U_\alpha[z]\subset  O_z$, witnessing that $z$ is an accumulation point of the set $D'\subset D$. But this contradicts the choice of $D$ as a closed discrete subspace in $X$.

Now we prove that the uniformity $\U$ coincides with the finest uniformity $\U(X)$ of the topological space $X$. It suffices to prove that every surjective continuous map $f:X\to M$ to a metric space $M$ is uniformly continuous with respect to the uniformity $\U$. Since $X$ is a $P$-space, the map $f$ remains continuous with respect to the discrete topology on $M$. So, we can assume that the metric space $M$ is discrete. Since the space $X$ has countable extent, the discrete metric space $M=f(X)$ has countable extent too and hence $M$ is countable. In this case the uniform continuity of the map $f$ was proved by Comfort and Ross in \cite{CR66}.
Hence the finest uniformity $\U(X)$ coincides with $\U$. Since the space  $X$ has countable extent, the uniformity $\U=\U(X)$ is $\w$-narrow.

Finally we show that under $\w_1=\mathfrak b$ the uniformity $\U=\U(X)$ has a $\mathfrak G$-base.
 By \cite{LPT}, there exists a monotone cofinal map $\varphi\colon \w^\w\to\mathfrak b=\w_1$.  For convenience of the reader, we present a simple construction of such map $\varphi$.
By the definition of the cardinal $\mathfrak b$, there exists a transfinite sequence $\{f_\alpha\}_{\alpha\in\mathfrak b}\subset\w^\w$ such that for every $f\in \w^\w$ there exists $\alpha\in\mathfrak b$ with $f_\alpha\not\le^* f$ (which means that the set $\{n\in\w:f(n)>g(n)\}$ is infinite). Then the map $\varphi:\w^\w\to\mathfrak b$, $\varphi:f\mapsto\min\{\alpha\in\mathfrak b:f_\alpha\not\le^* f\}$, is monotone and cofinal, witnessing that $\w^\w\succcurlyeq\mathfrak b=\w_1$. Then $(U_{\varphi(\alpha)})_{\alpha\in\w^\w}$ is a $\mathfrak G$-base of the uniformity $\U$.

\end{proof}

\section{Free locally convex spaces over uniform spaces}\label{s:Lc}

For a uniform space $X$ its {\em free locally convex space} is a pair $(\Lc_u(X),\delta_X)$ consisting of a locally convex space $\Lc_u(X)$ and a uniformly continuous map $\delta_X:X\to \Lc_u(X)$ such that for every uniformly continuous map $f:X\to Y$ into a locally convex space $Y$ there exists a continuous linear operator $\bar f:\Lc_u(X)\to Y$ such that $\bar f\circ\delta_X=f$. By a standard technique
\cite{Rai} it can be shown that for every uniform space $X$ the free locally convex space $(\Lc_u(X),\delta_X)$ exists and is unique up to a topological isomorphism.

Observe that for a topological space $X$ the free locally convex space $(\Lc(X),\delta_X)$ coincides with the free locally convex space $(\Lc_u(X),\delta_X)$ of the space $X$ endowed with the finest uniformity $\U(X)$.

For a uniform space $X$ by $C_u(X)$ we denote the linear space of all uniformly continuous real-valued functions on $X$. A subset $E\subset C_u(X)$ is called
\begin{itemize}
\item {\em pointwise bounded} if for every point $x\in X$ the set $E(x)=\{f(x):f\in E\}$ is bounded in the real line;
\item {\em equicontinuous} if for every $\e>0$ there is an entourage $U\in\U(X)$ such that $|f(x)-f(y)|<\e$ for all $f\in E$ and $(x,y)\in U$.
\end{itemize}

By $\E(C_u(X))$ we denote the set of pointwise bounded equicontinuous subfamilies in $C_u(X)$.
The set $\E(C_u(X))$ is a poset with respect to the inclusion order (defined by $E\le F$ iff $E\subset F$).

Let $C^*_u(X)$ be the linear space of all linear functionals on $C_u(X)$, endowed with the topology of uniform convergence on pointwise bounded equicontinuous subsets of $C_u(X)$. A neighborhood base of this topology at zero consists of the polar sets
$${E}^*=\{\mu\in C^*_u(X):\sup_{f\in E}|\mu(f)|\le1\}\mbox{ \ where $E\in\E(C_u(X))$}.$$

Each element $x\in X$ can be identified with the Dirac measure $\delta_x\in C_u^*(X)$ assigning to each function $f\in C_u(X)$ its value $f(x)$ at $x$. It can be shown that the map $\delta:X\to C^*_u(X)$, $\delta:x\mapsto\delta_x$, is uniformly continuous in the sense that for any neighborhood $V\subset C_u^*(X)$ of zero there is an entourage $U\in\U(X)$ such that $\delta_x-\delta_y\in V$ for any $(x,y)\in U$.

Let $\Lc_u(X)$ be the linear hull of the set $\delta(X)=\{\delta_x:x\in X\}$ in $C_u^*(X)$. By \cite{Rai}, the pair $(\Lc_u(X),\delta)$ is a free locally convex space over the uniform space $X$.

In the space $\Lc_u(X)$ consider the subspaces
$$
\begin{aligned}
&\IN(X{-}X)=\{n(\delta_x-\delta_y):n\in\IN,\;x,y\in X\},\\
&\tfrac1\IN X=\{\tfrac1n\delta_x:n\in\IN,\;x\in X\}, \mbox{ and}\\
&\Lambda_u(X)= \IN(X{-}X)\cup\tfrac1\IN X.
\end{aligned}
$$

\begin{lemma}\label{l:u1} For any uniform space $X$ we have the reductions
$$\E(C_u(X))\succcurlyeq \Tau_0(C^*_u(X))\succcurlyeq \Tau_0(\Lc_u(X))\succcurlyeq \Tau_0(\Lambda_u(X))\succcurlyeq\E(C_u(X)),$$
therefore all posets above are equivalent each other.
\end{lemma}

\begin{proof} The monotone cofinal map $\E(C_u(X))\to \Tau_0(C_u^*(X))$,
which assigns to $E$ the polar set  ${E}^*$,
 witnesses that $\E(C_u(X))\succcurlyeq \Tau_0(C^*_u(X))$. The reductions $\Tau_0(C^*_u(X))\succcurlyeq \Tau_0(\Lc_u(X))\succcurlyeq \Tau_0(\Lambda_u(X))$ trivially follow from the inclusions $$\Lambda_u(X)\subset \Lc_u(X)\subset C_u^*(X).$$

To prove that  $\Tau_0(\Lambda_u(X))\succcurlyeq\E(C_u(X))$, for every neighborhood $V\subset \Lambda_u(X)$ of zero, consider the set
$$E_V=\bigcap_{\mu\in V}\{f\in C_u(X):|\mu(f)|\le 1\}.$$
First we show that the set $E_V$ is pointwise bounded and equicontinuous.

To see that $E_V$ is pointwise bounded, fix any point $x\in X$. Taking into account that the sequence $\big(\frac1n\delta_x\big)_{n\in\IN}$ converges to zero in $\Lambda_u(X)$, find a number $n\in\IN$ with $\frac1n\delta_x\in V$. Then for every $f\in E_V$ we get $|\frac1n\delta_x(f)|\le 1$ and hence $|f(x)|\le n$, which means that the set $E_V$ is pointwise bounded.

Next, we check that $E_V$ is equicontinuous. Given any point $x\in X$ and $\e>0$ we should find an entourage $U\in\U(X)$ such that $|f(x)-f(y)|<\e$ for all $f\in E_V$ and $(x,y)\in U$.
Find $n\in\IN$ such that $\frac1n<\e$. The uniform continuity of the map $n\delta:X\to C^*_u(X)$, $n\delta:x\mapsto n\cdot\delta_x$,  yields an entourage $U\in\U(X)$ such that $n(\delta_x-\delta_y)\in V$ for any $(x,y)\in U$. Then for any $f\in E_V$ the inclusion $n(\delta_x-\delta_y)\in V$ implies $|n(\delta_x(f)-\delta_y(f))|\le 1$ and  $|f(x)-f(y)|\le \frac1n<\e$, which means that the set $E_V$ is equicontinuous.

So, the map $\Tau_0(\Lambda_u(X))\to \E(C_u(X))$, $V\mapsto E_V$, is well-defined. It is clear that this map is monotone. It remains to show that it is cofinal. Given any pointwise bounded equicontinuous set $E\in\E(C_u(X))$, consider the neighborhood $V=\{\mu\in \Lambda_u(X):\sup_{f\in E}|\mu(f)|\le 1\}$ of zero in $\Lambda_u(X)$ and observe that $E\subset E_V$. So, $\Tau_0(\Lambda_u(X))\succcurlyeq \E(C_u^*(X))$.
\end{proof}

As before, in the following lemma we consider the space $C_u(X)$ as a poset endowed with the partial order $f\le g$ iff $f(x)\le g(x)$ for all $x\in X$.

\begin{lemma}\label{l:u2} For any uniform space $X$ of density $\kappa$ we have the reductions $$\U(X)^\w\times \w^\kappa\succcurlyeq\E(C_u(X))\cong(\U(X))^\w\times C_u(X).$$
\end{lemma}

\begin{proof} Fix a dense subset $D\subset X$ of cardinality $\kappa$. To see that $\U(X)^\w\times \w^D\succcurlyeq \E(C_u(X))$, for any pair $P=\big((U_n)_{n\in\w},\varphi)\in \U(X)^\w\times \w^D$ consider the pointwise bounded equicontinuous set
$$E_P=\bigcap_{x\in D}\{f\in C_u(X)\colon |f(x)|\le\varphi(x)\}\cap\bigcap_{n\in\w}\bigcap_{x,y\in U_n}\{f\in C_u(X)\colon |f(x)-f(y)|\le \tfrac1{2^{n}}\}.$$
It is clear that the map $\U(X)^\w\times \w^D\to \E(C_u(X))$, $P\mapsto E_P$, is monotone and cofinal, witnessing that $\U(X)^\w\times \w^D\succcurlyeq \E(C_u(X))$.

By analogy we can prove that $\U(X)^\w\times C_u(X)\succcurlyeq \E(C_u(X))$.

To see that $\E(C_u(X))\succcurlyeq \U(X)^\w\times C_u(X)$, to every pointwise bounded equicontinuous set $E\in \E(C_u(X))$ assign the function $$\varphi_E\in C_u(X),\;\;\varphi_E:x\mapsto\sup_{f\in E}|f(x)|,$$ and the sequence $(E_n)_{n\in\w}$ of entourages
$$E_n=\bigcap_{f\in E}\{(x,y)\in X\times X:|f(x)-f(y)|<2^{-n}\}\in\U(X),\;\;n\in\w.$$
It is clear that the map $\E(C_u(X))\succcurlyeq \U(X)^\w\times C_u(X)$, $E\mapsto ((E_n)_{n\in\w},\varphi_E)$, is well-defined and monotone. It remains to prove that this map is cofinal.

Fix any pair $P=\big((U_n)_{n\in\w},\varphi)\in \U(X)^\w\times C_u(X)$. Let $V_{-1}=X\times X$ and choose a sequence of entourages $(V_n)_{n\in\w}\in\U(X)^\w$ such that $V_n\circ V_n\circ V_n\subset U_n\cap V_{n-1}$ for all $n\in\w$. By Theorem~8.1.10 \cite{Eng}, there is a pseudometric $d:X\times X\to[0,1]$ such that
$$\{(x,y)\in X\times X:d(x,y)<2^{-n}\}\subset V_n\subset \{(x,y)\in X\times X:d(x,y)\le 2^{-n}\}$$for all $n\in\w$.
For every $z\in X$ consider the uniformly continuous function $f_z\in C_u(X)$ defined by $f_z(x)=\varphi(x)+d(z,x)$ for $x\in X$. It is easy to see that the function family $E=\{f_z:z\in X\}$ is pointwise bounded and equicontinuous. We claim that $((U_n)_{n\in\w},\varphi)\le ((E_n)_{n\in\w},\varphi_E)$. It is clear that $\varphi\le\varphi_E\le \varphi+1$. Next, we show that for every $n\in\w$ the set $E_n=\bigcap_{f\in E}\{(x,y)\in X\times X:|f(x)-f(y)|<2^{-n}\}\subset U_n$. Indeed, for any $(x,y)\in E_n$ we get $d(x,y)=|f_x(x)-f_x(y)|<2^{-n}$ and hence $(x,y)\in V_n\subset U_n$ by the choice of the pseudometric $d$ and the entourage $V_n$.
\end{proof}

Lemmas~\ref{l:u1} and \ref{l:u2} imply the following reduction theorem.

\begin{theorem}\label{t:u} For any uniform space $X$ of density $\kappa$ we have the reductions:
$$\Tau_0(C_u^*(X))\cong \Tau_0(\Lc_u(X))\cong \Tau_0(\Lambda_u(X))\cong \E(C_u(X))\cong \U(X)^\w\times C_u(X)\preccurlyeq \U(X)^\w\times \w^\kappa.$$
\end{theorem}

As an immediate consequence, we obtain a key characterization of uniform spaces $X$ whose free locally convex spaces $\Lc_u(X)$ have local  $\mathfrak G$-bases.

\begin{corollary}\label{c:Lu3} For a uniform space $X$ the following conditions are equivalent:
\begin{enumerate}
\item the free locally convex space $\Lc_u(X)$ has a local $\mathfrak G$-base;
\item the uniformity $\U(X)$ has a $\mathfrak G$-base and the poset $C_u(X)$ is $\w^\w$-dominated;
\item the uniformity $\U(X)$ has a $\mathfrak G$-base and the poset $C_u(X)$ is $\w^\w$-dominated in $\IR^X$.
\end{enumerate}
\end{corollary}

\begin{proof} The equivalence $(1)\Leftrightarrow(2)$ follows from the reduction $\Tau_0(\Lc_u(X))\cong C_u(X)\times\U(X)^\w$ established in Theorem \ref{t:u}. The implication $(2)\Ra(3)$ is trivial and $(3)\Ra(2)$ can be proved by analogy with Lemma~\ref{l:u2}.
\end{proof}

\begin{corollary}\label{c:sepun} For a separable uniform space $X$ the free locally convex space $\Lc_u(X)$ has a local $\mathfrak G$-base if and only if the uniformity $\U(X)$ of $X$ has a $\mathfrak G$-base.
\end{corollary}

A similar characterization holds for $\w$-narrow uniform spaces under the set-theoretic assumption $\w_1<\mathfrak b$.

\begin{corollary} Assume that $\w_1<\mathfrak b$.  For an $\w$-narrow uniform space $X$ the free locally convex space $\Lc_u(X)$ has a local $\mathfrak G$-base if and only if the uniformity $\U(X)$ of $X$ has a $\mathfrak G$-base.
\end{corollary}

\begin{proof} The ``only if'' part follows from Corollary~\ref{c:Lu3}. To prove the ``if'' part, assume that the uniformity $\U(X)$ has a $\mathfrak G$-base. By Theorem~\ref{t:ws}, the space $X$ is weakly cosmic and hence separable. Applying Theorem~\ref{t:u}, we get the reduction $\w^\w\cong \w^\w\times(\w^\w)^\w\succcurlyeq \w^\w\times \U(X)^\w\succcurlyeq \Tau_0(\Lc_u(X))$, which means that the space $\Lc_u(X)$ has a local $\mathfrak G$-base.
\end{proof}

\section{Free topological vector spaces over uniform spaces}\label{s:Lin}

For a uniform space $X$ its {\em free topological vector space} is a pair $(\Lin_u(X),\delta_X)$ consisting of a topological vector space $\Lin_u(X)$ and a uniformly continuous map $\delta_X:X\to \Lin_u(X)$ such that for every uniformly continuous map $f:X\to Y$ into a  topological vector space $Y$ there exists a continuous linear operator $\bar f:\Lin_u(X)\to Y$ such that $\bar f\circ\delta_X=f$.

Observe that for a topological space $X$ its free topological vector space $(\Lin(X),\delta_X)$ coincides with the free  topological vector space $(\Lin_u(X),\delta_X)$ of the space $X$ endowed with the finest uniformity $\U(X)$.

Given a uniform space $X$ we shall identify the free topological vector space $\Lin_u(X)$ of $X$ with the linear hull $L(X)$ of the set $\delta(X)=\{\delta_x:x\in X\}\subset C^*_u(X)$, endowed with the strongest topology turning $L(X)$ into a topological vector space and making the map $\delta:X\to L(X)$, $\delta:x\mapsto\delta_x$, uniformly continuous.

The topology of the space $\Lin_u(X)$ can be described as follows.
For a function $\varphi(x) \in \IR^X$ denote by $\hat{\varphi}:X\to[1,\infty)$ the function defined by $\hat\varphi(x) = \max\{1,\varphi(x)\}$ for $x\in X$.
 Given a pair  $\big((U_n)_{n\in\w},(\varphi_n)_{n\in\w}\big)\in\U(X)^\w\times \IR^X$, consider the sets
$$
\begin{aligned}
&\sum_{n\in\w}U_n=\bigcup_{m\in\w}\Big\{\sum_{n=0}^mt_n(\delta_{x_n}-\delta_{y_n}):|t_n|\le 1,\;(x_n,y_n)\in U_n\mbox{ for all }n\le m\Big\}\mbox{ and }\\
&\sum_{n\in\w}\tfrac1{\varphi_n} X=\bigcup_{m\in\w}\Big\{\sum_{n=0}^mt_n\delta_{x_n}:x_n\in X,\;
|t_n|\le
1/\hat{\varphi}_n(x)
\Big\}.
\end{aligned}
$$

We recall that for a uniform space $X$ by $C_\w(X)$ we denote the subset of $\IR^X$ consisting of $\w$-continuous functions $f:X\to\IR$.
A subset $D\subset \IR^X$ is called {\em directed} if for any functions $f,g\in D$ there is a function $h\in D$ such that $h\ge\max\{f,g\}$. We shall say that $D\subset \IR^X$ {\em dominates} $C_\w(X)$ if for every function $f\in C_\w(X)$ there exists a function $g\in D$ with $g\ge f$.

\begin{theorem}\label{t:lin-top} Let $X$ be a uniform space and $D\subset\IR^X$ be a directed subset dominating the set $C_\w(X)$ of $\w$-continuous functions on $X$. The family
$$\mathcal B=\Big\{\sum_{n\in\w}U_n+\sum_{n\in\w}\tfrac1{\varphi_n} X:(\U_n)_{n\in\w}\in\U(X)^\w,\;(\varphi_n)_{n\in\w}\in D^\w\Big\}$$is a neighborhood base at zero of the topology of the free topological vector space $\Lin_u(X)$. Consequently, $\U(X)^\w\times D^\w\succcurlyeq \Tau_0(\Lin_u(X))$.
\end{theorem}

\begin{proof} Given a pair $P=\big((U_n)_{n\in\w},(\varphi_n)_{n\in\w}\big)\in\U(X)^\w\times D^\w$ denote the set $$\sum_{n\in\w}U_n+\sum_{n\in\w}\tfrac1{\varphi_n} X$$ by $V_P$. Let $\tau$ be the topology on $L(X)$ consisting of the sets $W\subset L(X)$ such that for any $w\in W$ there is a pair $P\in\U(X)^\w\times D^\w$ such that $w+V_P\subset W$. The definition of the set $V_P$ implies that it belongs to the topology $\tau$.

We claim that the topology $\tau$ turns the vector space $L(X)$ into a topological vector space. We need to check the continuity of the addition $L(X)\times L(X)\to L(X)$, $(u,v)\mapsto u+v$, and the multiplication $\IR\times L(X)\to L(X)$, $(t,u)\mapsto tu$, with respect to the topology $\tau$. Since the topology $\tau$ is invariant under shifts, it suffices to check the continuity of the addition at zero. Given a neighborhood $W\in\tau$ of zero, find a pair $P=\big((U_n)_{n\in\w},(\varphi_n)_{n\in\w}\big)\in\U(X)^\w\times D^\w$ such that such that $V_P\subset W$. For every $n\in\w$ consider the entourage $U'_n=U_{2n}\cap U_{2n+1}\in \U(X)$ and choose a function $\psi_n\in D$ such that $\psi_n\ge \max\{\varphi_{2n},\varphi_{2n+1}\}$. We claim that for the pair $Q=\big((U_n')_{n\in\w},(\psi_n)_{n\in\w}\big)$ we get $V_Q+V_Q\subset V_P$. Indeed,
\begin{multline*}
V_Q+V_Q=\sum_{n\in\w}\frac1{\psi_n} X+\sum_{n\in\w}U_n'+\sum_{n\in\w}\frac1{\psi_n} X+\sum_{n\in\w}U_n'\subset\\
\subset\sum_{n\in\w}\frac1{\varphi_{2n}}X+\sum_{n\in\w}\frac1{\varphi_{2n+1}}X
+\sum_{n\in\w}U_{2n}+\sum_{n\in\w}U_{2n+1}=\sum_{n\in\w}\frac1{\varphi_n} X+\sum_{n\in\w}U_n=V_P.
\end{multline*}
The definition of the set $V_Q$ implies that it is open in the topology $\tau$. So, $V_Q$ is an open neighborhood of zero such that $V_Q+V_Q\subset V_P\subset W$, which proves the continuity of the addition at zero.

Next, we prove the continuity of the multiplication by a scalar.
Fix any pair $(t,u)\in\IR\times L(X)$ and any neighborhood $W\in\tau$ of the product $tu$.  Find $m\in\w$ such that $|t|\le m$ and $u=\sum_{n=0}^mt_n\delta_{x_n}$ for some real numbers $t_0,\dots,t_n$ and some points $x_0,\dots,x_n\in X$. By the continuity of the addition at zero, there exists a neighborhood $W_0\in\tau$ of zero such that the set $\sum_{n=0}^{m+1}W_0=\big\{\sum_{n=0}^{m+1}w_n:w_0,\dots,w_{m+1}\in W_0\big\}$ is contained in the neighborhood $W-tu$ of zero. Find a pair $Q=\big((U_n')_{n\in\w},(\psi_n)_{n\in\w}\big)\in\U(X)^\w\times D^\w$ such that $V_Q\subset W_0$. Choose $\e\in(0,1]$ so small that $\e t_n\le 1/\hat{\psi}_n(x_n)$ for all $n\le m$, which implies that $[-\e,\e]\cdot u\subset V_Q$. We claim that $(t-\e,t+\e)\cdot (u+V_Q)\subset W$. Indeed, take any pair $(t',u')\in(t-\e,t+\e)\times (u+V_Q)$ and observe that
\begin{multline*}
t'u'-tu=(t{+}t'{-}t)(u{+}u'{-}u)-tu=t(u'{-}u)+(t'{-}t)u+(t'{-}t)(u'{-}u)\in\\
\in tV_Q+[-\e,\e]u+[-1,1]V_Q\subset mV_Q+V_Q+V_Q\subset \sum_{n=0}^{m+1}W_0\subset W-tu
\end{multline*}
and hence $t'u'\in W$.

Therefore the topology $\tau$ turns the linear hull $L(X)$ of the set $\{\delta_x:x\in X\}$ into a topological vector space. The definition of the topology $\tau$ implies that the map $\delta:X\to L(X)$, $\delta:x\mapsto\delta_x$, is uniformly continuous. Now the definition of the free  topological vector space $\Lin_u(X)$ guarantees that the identity map $\Lin_u(X)\to L(X)$ is continuous (with respect to the topology $\tau$). We claim that this map is a homeomorphism.

Given any neighborhood $W\subset \Lin_u(X)$ of zero, we should find a pair $P\in\U(X)^\w\times D^\w$ such that $V_P\subset W$.
Consider the constant map $\mathbf 1:X\to\{1\}\subset\IR$ and the induced linear continuous operator $\bar{\mathbf 1}:\Lin_u(X)\to\IR$. Replacing $W$ by a smaller neighborhood of zero, we  can assume that $\bar{\mathbf 1}(W)\subset(-1,1)$. By an argument similar to that of \cite[\S 1.2]{Rolewicz}, we can find an invariant continuous pseudometric $d$ on $\Lin_u(X)$ such that $\{x\in \Lin_u(X):d(x,0)<1\}\subset W$ and $d(tx,0)\le d(x,0)$ for any $t\in[-1,1]$ and $x\in X$. For every $n\in\w$ consider the entourage $U_n=\{(x,y)\in X\times X:d(x,y)<2^{-n-2}\}$. Using the paracompactness of the pseudometric space $(\Lin_u(X),d)$, for every $n\in\w$ it is easy to construct a $d$-continuous function $\psi_n:\Lin_u(X)\to[1,\infty)$ such that $d(tx,0)<2^{-n-2}$ for every $x\in X$ and $t\in\IR$ with $|t|\le 1/\varphi_n(x)$. The uniform continuity of the map $\delta:X\to \Lin_u(X)$ and the $d$-continuity of the map $\psi_n$ implies the $\w$-continuity of the composition $\psi_n\circ\delta:X\to\IR$. Finally, choose a function $\varphi_n\in D$ such that $\varphi_n\ge \psi_n\circ\delta$.
It follows that the pair $P=\big((U_n)_{n\in\w},(\varphi_n)_{n\in\w}\big)$ belongs to the poset $\U(X)^\w\times D^\w$. Using the triangle inequality it is easy to check that $V_P\subset W$.

Therefore the correspondence $\U(X)^\w\times D^\w\to \Tau_0(\Lin_u(X))$, $P\mapsto V_P$, is monotone and cofinal, yielding the reduction $\U(X)^\w\times D^\w\succcurlyeq \Tau_0(\Lin_u(X))$.
\end{proof}

\begin{theorem}\label{t:Lin+Lc-reduct} For any uniform space $X$ and any directed subset $D\subset \IR^X$ dominating $C_\w(X)$ we get the reductions
$$\U(X)^\w\times D^\w\succcurlyeq \Tau_0(\Lin_u(X))\succcurlyeq \Tau_0(\Lc_u(X))\cong \U(X)^\w\times C_u(X).$$
\end{theorem}

\begin{proof} The reduction $\U(X)^\w\times D^\w\succcurlyeq \Tau_0(\Lin_u(X))$ was proved in Theorem~\ref{t:lin-top}.
To see that $\Tau_0(\Lin_u(X))\succcurlyeq \Tau_0(\Lc_u(X))$, consider the monotone map $\Tau_0(\Lin_u(X))\to\Tau_0(\Lc_u(X))$ assigning to each neighborhood $U\in \Tau_0(\Lin_u(X))$ of zero its convex hull $\conv(U)$. This map is clearly well-defined.
Its cofinality follows from \cite[Proposition 5, TVS II.27]{Bourbaki}.
 The final reduction  $\Tau_0(\Lc_u(X))\cong \U(X)^\w\times C_u(X)$ was proved in Theorem~\ref{t:u}.
 \end{proof}

\begin{corollary}\label{c:Lin-dom} Let $X$ be a uniform space
whose uniformity has a $\mathfrak G$-base. If the set $C_\w(X)$ is $\w^\w$-dominated in $\IR^X$, then the free topological vector space $\Lin_u(X)$ has a local $\mathfrak G$-base.
\end{corollary}

\begin{proof} If the set $C_\w(X)$ is $\w^\w$-dominated in $\IR^X$, then we can find a set $D=\{f_\alpha\}_{\alpha\in\w^\w}\subset\IR^X$ which dominates $C_\w(X)$ and such that $f_\alpha\le f_\beta$ for all $\alpha\le\beta$ in $\w^\w$. It follows that $D$ is a directed set with $\w^\w\succcurlyeq D$. Applying Theorem~\ref{t:Lin+Lc-reduct}, we get the reductions
$$\w^\w\cong (\w^\w)^\w\times(\w^\w)^\w\succcurlyeq D^\w\times \U(X)^\w\succcurlyeq\Tau_0(\Lin_u(X)),$$witnessing that the space $\Lin_u(X)$ has a local $\mathfrak G$-base.
\end{proof}

\section{Uniform spaces with a $\mathfrak G$-base}\label{s:U}

In this section we establish some topological properties of uniform spaces whose uniformity has a $\mathfrak G$-base. The main result of this section is the Metrizability Theorem~\ref{t:metr}.
 We start with the following fact proved in \cite{LPT}.

\begin{proposition}\label{p:sigma'} The finest uniformity $\U(X)$ of a metrizable topological space $X$ has a $\mathfrak G$-base if the set $X'$ of non-isolated points of $X$ is $\sigma$-compact.
\end{proposition}

Let $X$ be a uniform space whose uniformity $\U(X)$ has a $\mathfrak G$-base  $(U_\alpha)_{\alpha\in\w^\w}$.

Let $\w^{<\w}=\bigcup_{n\in\w}\w^n$ be the family of finite sequences of finite ordinals. Observe that for every $\alpha\in\w^\w$ and $n\in\w$ the restriction $\alpha|n=(\alpha(0),\dots,\alpha(n-1))$ belongs to $\w^{<\w}$. For a finite sequence $\alpha\in\w^n\subset \w^{<\w}$ let ${\uparrow}\alpha=\{\beta\in\w^\w:\alpha=\beta|n\}\subset\w^\w$ and $$U_\alpha=\bigcap_{\beta\in{\uparrow}\alpha}U_\beta\subset X\times X.$$

\begin{lemma}\label{l:countable} Let $(F_n)_{n\in\w}$, $(E_n)_{n\in\w}$ be two sequences of finite subsets of a uniform space $X$ such that for every entourage $U\in\U(X)$ there exists $k\in\w$ such that $(F_n\times E_n)\cap U\ne\emptyset$ for every $n\ge k$.  If $(U_\alpha)_{\alpha\in\w^\w}$ is a $\mathfrak G$-base of the uniformity $\U(X)$ of $X$, then for every $\alpha\in\w^\w$ there is $k\in\w$ such that  $(F_n\times E_n)\cap U_{\alpha|k}\ne\emptyset$ for every $n\ge k$.
\end{lemma}

\begin{proof} To derive a contradiction, assume that for every $k\in\w$ there is a number $n_k\ge k$ such that $(F_{n_k}\times E_{n_k})\cap U_{\alpha|k}=\emptyset$. By the definition of the entourage $U_{\alpha|k}$, for every pair $(x,y)\in F_{n_k}\times E_{n_k}$ there is a function $\beta_{k,x,y}\in\w^\w$ such that $\beta_{k,x,y}|k=\alpha|k$ and $(x,y)\notin U_{\beta_{k,x,y}}$. Consider the function $\beta_k=\max\{\beta_{k,x,y}:(x,y)\in F_{n_k}\times E_{n_k}\}$ and observe that the inclusion $U_{\beta_k}\subset U_{\beta_{k,x,y}}$ for $(x,y)\in F_{n_k}\times E_{n_k}$ implies that $(F_{n_k}\times E_{n_k})\cap U_{\beta_k}=\emptyset$. It follows that $\beta_k|k=\alpha|k$ and hence $\beta_k(k-1)=\alpha(k-1)$ if $k\ge 1$. Let $\beta\in\w^\w$ be the function defined by $\beta(k)=\max\{\beta_i(k):i\le k+1\}$. We claim that $\beta\ge\beta_k$ for every $k\in\w$. Fix any number $n\in\w$. If $k\le n+1$, then $\beta(n)=\max\{\beta_i(n):i\le n+1\}\ge \beta_k(n)$. If $k>n+1$, then $\beta(n)\ge \beta_{n+1}(n)=\alpha(n)=\beta_k(n)$.
For every $k\in\w$ the inequality $\beta\ge\beta_k$ implies the inclusion $U_\beta\subset U_{\beta_k}$ and hence $(F_{n_k}\times E_{n_k})\cap U_\beta\subset (F_{n_k}\times E_{n_k})\cap U_{\beta_k}=\emptyset$, which contradicts the choice of the sequences $(E_n)_{n\in\w}$ and $(F_n)_{n\in\w}$.
\end{proof}

For a subset $A$ of a topological space $X$ by $\bar A$ and $A^\circ$ we denote the closure and the interior of $A$ in $X$, respectively.
According to \cite {KocSch}, a topological space $X$ is said to have the {\em Reznichenko property} at a point $x\in X$,
if for any subset $A \subset X$ with $x\in \bar A$,
there exists a sequence $(F_n)_{n\in\w}$ of pairwise disjoint finite sets $F_n\subset A$ such that each neighborhood $O_x\subset X$ of $x$ intersects all but finitely many sets $F_n$, $n\in\w$.

We shall say that a topological space $X$ has the {\em strong} ({\em open-}){\em Reznichenko property} at a point $x\in X$ if for any decreasing sequence $(U_n)_{n\in\w}$ of (open) sets in $X$ with $x\in\bigcap_{n\in\w}\bar U_n$ there exists a sequence $(F_n)_{n\in\w}$ of finite sets $F_n\subset U_n$ such that each neighborhood $O_x\subset X$ intersects all but finitely many sets $F_n$, $n\in\w$.

By \cite{BaR}, a (regular) topological space $X$ is first countable at a point $x\in X$ if and only if $X$ has a $\mathfrak G$-base at $x$ and $X$ is strong (open-)Reznichenko at $x$.  More information on strong (open-)Reznichenko spaces can be found in \cite{BaR}.

\begin{lemma}\label{l:neib} Let $X$ be a uniform space and $(U_\alpha)_{\alpha\in\w^\w}$ be a $\mathfrak G$-base of its uniformity. If the space $X$ is strong open-Reznichenko at some point $x\in X$, then for every $\alpha\in\w^\w$ there exists $k\in\w$ such that $\bar U_{\alpha|k}[x]$ is a neighborhood of $x$.
\end{lemma}

\begin{proof} Assume that for every $k\in\w$ the closed set $\bar U_{\alpha|k}[x]$ is not a neighborhood of $x$. Then the open set $A_k=X\setminus\bar U_{\alpha|k}(x)$ contains the point $x$ in its closure. Since $X$ is strong open-Reznichenko at $x$, there exists a sequence $(E_k)_{k\in\w}$ of finite sets $E_k\subset A_k$ such that every neighborhood of $x$ intersects all but finitely many sets $E_k$, $k\in\w$. By Lemma~\ref{l:countable}, there exists $k\in\w$ such that $U_{\alpha|k}[x]\cap E_n\ne\emptyset$ for all $n\ge k$. But this contradicts the choice of the set $E_k\subset A_k=X\setminus \bar U_{\alpha|k}[x]\subset X\setminus U_{\alpha|k}[x]$.
\end{proof}

A topological space $X$ is called {\em closed-$\bar G_\delta$} if for each closed subset $F\subset X$ there exists a sequence $(W_n)_{n\in\w}$ of open sets in $X$ such that $F=\bigcap_{n\in\w}W_n=\bigcap_{n\in\w}\overline{W}_n$. It is easy to see that each perfectly normal space is closed-$\bar G_\delta$. We remind that a topological space $X$ is called  {\em perfectly normal} if $X$ is normal and every closed set is a $G_\delta$-set in $X$.


\begin{theorem}\label{t:metr} A topological space $X$ is metrizable if and only if $X$ is strongly open-Reznichenko, closed-$\bar G_\delta$, and the topology of $X$ is generated by a uniformity  possessing a $\mathfrak G$-base.
\end{theorem}

\begin{proof} The ``only if'' part is trivial. To prove the ``if'' part, assume that the space $X$ is strongly open-Reznichenko, closed-$\bar G_\delta$, and the topology of $X$ is generated by a uniformity $\U$ possessing a $\mathfrak G$-base $(U_\alpha)_{\alpha\in\w^\w}$.
The metrizaility of $X$ will be proved using the Moore Metrization Theorem \cite[5.4.2]{Eng} (see also \cite[1.4]{Grue}).

For every $\alpha\in \w^{<\w}$ consider the entourage $U_\alpha=\bigcap\{U_\beta:\beta\in {\uparrow}\alpha\}\subset X\times X$. For every $x\in X$ the entourage $U_\alpha$ determines the $U_\alpha$-ball $U_\alpha[x]=\{y\in X:(x,y)\in U_\alpha\}$. Let $\bar U_{\alpha}[x]$ be its closure in $X$ and $\bar U^\circ_\alpha[x]$ be the interior of $\bar U_\alpha[x]$ in $X$.
Then $\U_\alpha=\{\bar U^\circ_\alpha[x]:x\in X\}$ is a family of open subsets of $X$ and its union $\bigcup\U_\alpha$ is an open subset in $X$. Since the space $X$ is closed-$\bar G_\delta$,  $X\setminus\bigcup\U_\alpha=\bigcap_{m\in\w}\overline{W}_{\alpha,m}$ for some  sequence $(W_{\alpha,m})_{m\in\w}$ of open sets in $X$. For every $m\in\w$ consider the open cover $\U_{\alpha,m}=\U_\alpha\cup\{W_{\alpha,m}\}$ of $X$.

It follows that $\{\U_{\alpha,m}:\alpha\in\w^{<\w},\;m\in\w\}$ is a countable family of open covers of $X$. The metrizability of $X$ will follow from the Moore metrization Theorem \cite[5.4.2]{Eng} as soon as for every point $x\in X$ and  neighborhood $O_x\subset X$ of $x$ we find a neighborhood $V_x\subset X$ of $x$ and a pair $(\alpha,m)\in\w^{<\w}\times\w$ such that $\St(V_x,\U_{\alpha,m})\subset O_x$, where $\St(V_x,\U_{\alpha,m})=\bigcup\{U'\in\U_{\alpha,m}:V_x\cap U'\ne\emptyset\}$.

First, find an entourage $V\in\U$ such that $V[x]\subset O_x$ and choose an entourage $U\in\U$ such that $U^{-1}UU^{-1}UU^{-1}U\subset V$. Since $(U_\alpha)_{\alpha\in\w^\w}$ is a $\mathfrak G$-base of the uniformity $\U$, there is a sequence $\alpha\in\w^\w$ such that $U_\alpha\subset U\subset V$ and hence $U_\alpha[x]\subset V[x]\subset O_x$. By Lemma~\ref{l:neib}, for some $n\in\w$ the set $\bar U_{\alpha|n}[x]$ contains the point $x$ in its interior. So, $x\in \bar U^\circ_{\alpha|n}[x]\in\U_{\alpha|n}$ and hence $x\in \bigcup\U_{\alpha|n}\setminus \overline{W}_{\alpha|n,m}$ for some $m\in\w$. We claim that the neighborhood $V_x=\bar U^\circ_{\alpha|n}[x]\setminus \overline{W}_{\alpha|n,m}$ of $x$ and the cover $\U_{\alpha|n,m}$ have the required property: $\St(V_x,\U_{\alpha|n,m})\subset O_x$. Given any point $y\in \St(V_x,\U_{\alpha|n,m})$, we can find a point $v\in V_x$ and a set $U'\in\U_{\alpha|n,m}$ such that $\{y,v\}\subset U'$. Since $v\in V_x\subset X\setminus\overline{W}_{\alpha|n,m}$, the set $U'$ is equal to the set $\bar U^\circ_{\alpha|n}(z)$ for some point $z\in X$.
The inclusion $$v\in V_x\subset \bar U^\circ_{\alpha|n}[x]\subset \bar U_{\alpha|n}[x]\subset\bar U_\alpha[x]\subset \bar U[x]$$implies $U[v]\cap U[x]\ne\emptyset$ and hence $v\in U^{-1} U[x]$.
On the other hand, the inclusions
$$\{y,v\}\subset U'=\bar U^\circ_{\alpha|n}[z]\subset \bar U_{\alpha|n}[z]\subset \bar U_\alpha[z]\subset \bar U[z]$$ imply $U[y]\cap U[z]\ne \emptyset\ne U[z]\cap U[v]$,
$z\in U^{-1} U[v]$ and finally
$$y\in U^{-1} U[z]\subset U^{-1} U U^{-1} U[v]\subset U^{-1} U U^{-1} U U^{-1} U[x]\subset V[x]\subset O_x.$$
\end{proof}





\section{Some properties of the poset $C(X)$}\label{s:CX}

In this section we study properties of Tychonoff spaces $X$ for which the poset $C(X)$ is $\w^\w$-dominated (in $\IR^X$).

For a topological space $X$ by $C_p(X)$ we denote the subspace of Tychonoff power $\IR^X$ consisting of continuous functions.
A Tychonoff space $X$ is cosmic if and only if its function space $C_p(X)$ is cosmic (see Theorem I.1.3 in \cite{Arch}). By the Calbrix Theorem 9.7 in \cite{kak} (and the Christensen Theorem 9.6 in \cite{kak}), for a (separable metrizable) Tychonoff space $X$ the function space  $C_p(X)$ is analytic (if and) only if $X$ is $\sigma$-compact. A topological space $X$ is called {\em analytic} if $X$ is a continuous image of a Polish space.

By a ({\em compact}) {\em resolution} of a topological space $X$ we understand a family $(X_\alpha)_{\alpha\in\w^\w}$ of (compact) subsets of $X$ such that $X=\bigcup_{\alpha\in\w^\w}X_\alpha$ and $X_\alpha\subset X_\beta$ for every $\alpha\le\beta$ in $\w^\w$.
More information on compact resolutions can be found in the monograph \cite{kak}.

A topological space $X$ is called {\em $K$-analytic} if
$X=\bigcup_{\alpha\in\w^\w}K_\alpha$ for a family $(K_\alpha)_{\alpha\in\w^\w}$ of compact subsets of $X$, which is {\em upper semicontinuous} in the sense that for every open set $U\subset X$ the set $\{\alpha\in\w^\w:K_\alpha\subset U\}$ is open in the product topology of the space $\w^\w$. If each compact set $K_\alpha$ is a singleton $\{f(\alpha)\}$ then the upper semicontinuity of the family $(K_\alpha)_{\alpha\in\w}$ is equivalent to the continuity of the map $f:\w^\w\to X$, meaning that the space $X$ is analytic.

It is clear that each $K$-analytic space admits a compact resolution and is Lindel\"of. By Proposition 3.13 \cite{kak}, a Lindel\"of regular space $X$ is $K$-analytic if and only if $X$ has a compact resolution. By
 Talagrand's Proposition 6.3 \cite{kak}, a regular space $X$ is analytic if and only if $X$ is $K$-analytic and submetrizable. We recall that a topological space $X$ is {\em submetrizable} if $X$ admits a continuous metric (equivalently, $X$ admits a continuous injective map into a metrizable space).
 A topological space $X$ is {\em angelic} if each countably compact subset $K$ in $X$ has compact closure $\bar K$ in $X$ and for each point $x\in\bar K$ there exists a sequence $\{x_n\}_{n\in\w}\subset K$ convergent to $x$. A subset $K\subset X$ is called {\em countably compact} in $X$ if each sequence $\{x_n\}_{n\in\w}\subset K$ has an accumulation point in $X$.

A subset $B$ of a topological space $X$ is called {\em bounded} if for each continuous function $f:X\to\IR$ the image $f(B)$ is bounded in $\IR$. A topological space $X$ is called {\em $\sigma$-bounded} if $X$ can be written as the countable union of bounded sets in $X$.


\begin{theorem}\label{t:ACFK} For a topological space $X$ we have the implications\newline
$(1)\Ra(2)\Leftrightarrow(3)\Ra(4,5)$ of the following properties:
\begin{enumerate}
\item $X$ is $\sigma$-bounded;
\item the set $C(X)$ is $\w^\w$-dominated in $\IR^X$;
\item $C_p(X)$ is contained in a $K$-analytic subspace of $\IR^X$;
\item the space $C_p(X)$ is angelic;
\item each  metrizable image of $X$ is $\sigma$-compact.
\end{enumerate}
\end{theorem}

\begin{proof} To prove that $(1)\Ra(2)$, assume that $X=\bigcup_{n\in\w}B_n$ is a countable union of bounded sets $B_n$. For every $\alpha\in\w^\w$ consider the function $f_\alpha:X\to\IR$ such that $f_\alpha(x)=\alpha(n)$ for any $n\in\w$ and $x\in B_n\setminus\bigcup_{k<n}B_k$. It is easy to see that the monotone correspondence $\w^\w\to\IR^X$, $\alpha\mapsto f_\alpha$, witnesses that the set  $C(X)$ is $\w^\w$-dominated in $\IR^X$.



The implications $(2)\Leftrightarrow(3)\Ra(4)$ follow from Proposition 9.6 of \cite{kak} and the implication $(2)\Ra(5)$ is proved in Theorem~\ref{t:dominat}.
\end{proof}

\begin{remark}  By Example 2 of \cite{Lei}, there exists a Tychonoff space with a unique non-isolated point $X$ such that
$X$ is Lindel\"of, the function space $C_p(X)$ is $K$-analytic but $X$ is not $\sigma$-bounded. This example shows that the implication $(1)\Ra(2)\Leftrightarrow (3)$ in Theorem~\ref{t:ACFK} cannot be reversed and thus answers the corresponding problem posed by Arhangel'ski and Calbrix in \cite{AC} and then repeated in \cite[p.~215]{kak}.
\end{remark}

Metrizable spaces $X$ with $\w^\w$-dominated poset $C(X)$ can be characterized as follows.

\begin{theorem}\label{t:s->dom} For a metrizable space $X$ the following conditions are equivalent:
\begin{enumerate}
\item $X$ is $\sigma$-compact;
\item the poset $C(X)$ is $\w^\w$-dominated;
\item the set $C(X)$ is $\w^\w$-dominated in $\IR^X$;
\item $C(X)$ is contained in a $K$-analytic subspace of $\IR^X$;
\item $C_p(X)$ is analytic.
\end{enumerate}
\end{theorem}

\begin{proof} The equivalences $(1)\Leftrightarrow(3)\Leftrightarrow(4)\Leftrightarrow(5)$ are proved in Corollary 9.2 and Proposition 9.6 of the monograph \cite{kak}. The implication $(2)\Ra(3)$ is trivial. It remains to prove that $(1)\Ra(2)$. Write $X$ as the union $X=\bigcup_{n\in\w}X_n$ of an increasing sequence $(X_n)_{n\in\w}$ of compact subsets. Fix any metric $d$ generating the topology of $X$. For every $\alpha\in\w^\w$ let
 $$B_\alpha=\bigcap_{n\in\w}\big\{f\in \IR^X:\sup_{x\in X_n}|f(x)|\le\alpha(n)\big\}$$
and consider the pointwise bounded equicontinuous set
 $$E_\alpha=B_\alpha \cap\bigcap_{n\in\w}\bigcap_{x\in X_n}\big\{f\in\IR^X:\forall y\in X\;d(x,y)<\tfrac1{\alpha(n)}\Ra |f(x)-f(y)|<1/n\big\}$$in $C(X)$. It follows that the function $f_\alpha:X\to\IR$, $f_\alpha:x\mapsto \sup_{f\in E_\alpha}|f(x)|$ is continuous and the map $\w^\w\to C(X)$, $\alpha\mapsto f_\alpha$, witnesses that the function space $C(X)$ is $\w^\w$-dominated.
\end{proof}

\begin{proposition}\label{p:G+dom->K} If the finest uniformity $\U(X)$ of a topological space $X$ has a $\mathfrak G$-base and the poset $C(X)$ is $\w^\w$-dominated in $\IR^\w$, then the function space $C_p(X)$ is $K$-analytic. Consequently, all finite powers $X^n$, $n\in\IN$, of $X$ are countably tight.
\end{proposition}

\begin{proof} By Lemma~\ref{l:u2}, $\w^\w\succcurlyeq \U(X)^\w\times C(X)\cong \E(C(X))$. Consequently, there exists a monotone cofinal map $E:\w^\w\to\E(C(X))$ assigning to each $\alpha\in\w^\w$ a pointwise bounded equicontinuous set $E_\alpha\subset C(X)$. Then the closure $\bar E_\alpha$ of $E_\alpha$ in $\IR^X$ is compact and is contained in $C_p(X)$. It follows that $(\bar E_\alpha)_{\alpha\in\w^\w}$ is a compact resolution of $C_p(X)$. By Tkachuk's Theorem \cite{Tk} (see Theorem 9.3 in \cite[p.~203]{kak}), the function space $C_p(X)$ is $K$-analytic and hence Lindel\"of (see Proposition 3.4 in \cite[p.~66]{kak}).
By Asanov's Theorem (see Theorem I.4.1 in \cite{Arch}), finite powers of $X$ have countable tightness.
\end{proof}

Assuming additionally that the space $X$ is separable, we obtain a stronger result.

\begin{theorem}\label{t:G-base->sigma} If the finest uniformity $\U(X)$ of a separable Tychonoff space $X$ has a $\mathfrak G$-base, then the function space $C_p(X)$ is analytic and the space $X$ is cosmic and $\sigma$-compact, in other words, $X$ is a countable union of metrizable compact subspaces.
\end{theorem}

\begin{proof} By Lemma~\ref{l:u2}, $\w^\w\succcurlyeq \U(X)^\w\times \w^\w\cong \E(C(X))$. Consequently, there exists a monotone cofinal map $E:\w^\w\to\E(C(X))$ assigning to each $\alpha\in\w^\w$ a pointwise bounded equicontinuous set $E_\alpha\subset C(X)$. Then the closure $\bar E_\alpha$ of $E_\alpha$ in $\IR^X$ is compact and is contained in $C_p(X)$. It follows that $(\bar E_\alpha)_{\alpha\in\w^\w}$ is a compact resolution of $C_p(X)$. Applying as before  Tkachuk's Theorem \cite{Tk}, we get that the function space $C_p(X)$ is $K$-analytic. The separability of the space $X$ implies the submetrizability of the function space $C_p(X)$. By Talagrand's Proposition 6.3 \cite[p.~144]{kak}, the submetrizable $K$-analytic space $C_p(X)$ is analytic and hence cosmic. By Proposition 9.2 \cite[p.~207]{kak} and Calbrix's Theorem 9.7 \cite[p.~206]{kak}, the space $X$ is cosmic and $\sigma$-compact.
\end{proof}

Combining Theorem~\ref{t:G-base->sigma} with Theorem~\ref{t:ws} we get the following  result holding under the set theoretic assumption $\w_1<\mathfrak b$.

\begin{corollary}\label{c:w1<b+n} Assume that $\w_1<\mathfrak b$. If the finest uniformity $\U(X)$ of a Tychonoff space $X$ is $\w$-narrow and has a $\mathfrak G$-base, then the space $X$ is cosmic and $\sigma$-compact.
\end{corollary}

Combining Corollary~\ref{c:w1<b+n} with Corollary~\ref{c:narrow} we get another consistency result.

\begin{corollary}\label{c:w1<b} Assume that $\w_1<\mathfrak b$. If the finest uniformity $\U(X)$ of a Tychonoff space $X$ has a $\mathfrak G$-base and the set $C(X)$ is $\w^\w$-dominated in $\IR^X$, then the space $X$ is cosmic and $\sigma$-compact.
\end{corollary}

Example~\ref{ex:PU} shows that Corollary~\ref{c:w1<b+n} is not true in ZFC.

\begin{problem}\label{prob6.5} Is Corollary~\ref{c:w1<b}  true in ZFC?
\end{problem}

\begin{remark}\label{rem:notC} The $P$-space $X$ constructed in Example~\ref{ex:PU} does not provide a counterexample to Problem~\ref{prob6.5} since it has uncountable tightness, which implies that the set $C(X)$ is not $\w^\w$-dominated in $\IR^X$ (according to  Proposition~\ref{p:G+dom->K}).
\end{remark}


\section{Free (locally convex) topological vector spaces}\label{s:Lin+Lc}

In this section we apply the results of the preceding sections and characterize topological spaces $X$ whose free (locally convex) topological vector spaces have local $\mathfrak G$-bases.
Simultaneously, we shall characterize Tychonoff $k$-spaces $X$ whose double function spaces $C_k(C_k(X))$ have a $\mathfrak G$-base.

For a Tychonoff topological space $X$ by $C_k(X)$ we denote the space of  all continuous real-valued functions on $X$ endowed with the compact-open topology.

Following \cite{BG}, we define a topological space $X$ to be {\em Ascoli} if every compact subset $K$ of $C_k(X)$ is equicontinuous. This is equivalent to saying that the evaluation map $e:K\times X\to \IR$, $e:(f,x)\mapsto f(x)$, is jointy continuous.
By \cite{BG}, the class of Ascoli spaces contains all $k$-spaces and all Tychonoff $k_\IR$-spaces.

\begin{theorem}\label{t:Lin+Lc} Let $X$ be a Tychonoff space. Then
\begin{enumerate}
\item the following conditions are equivalent:
\begin{enumerate}
\item[$(\Lc)$] the free locally convex space $\Lc(X)$ of $X$ has a local $\mathfrak G$-base;
\item[$(\Lin)$] the free topological vector space $\Lin(X)$ of $X$ has a local $\mathfrak G$-base;
\item[$(\U \mathsf C)$] the finest uniformity $\U(X)$ of $X$ has a $\mathfrak G$-base and the poset $C(X)$ is $\w^\w$-dominated;
\item[$(\U \hat{\mathsf C})$] the uniformity $\U(X)$ has a $\mathfrak G$-base and $C(X)$ is $\w^\w$-dominated in $\IR^X$.
\end{enumerate}
\item
The conditions $(\Lc),(\Lin)$ imply that every metrizable continuous image of $X$ is $\sigma$-compact, the function space $C_p(X)$ is $K$-analytic, and  all finite powers of $X$ are countably tight.
\item
If the space $X$ is separable or $\w_1<\mathfrak b$, then the conditions  $(\Lc),(\Lin)$  imply that $X$ is a countable union of compact metrizable spaces.
\item If the space $X$ is separable, then the conditions $(\Lc),(\Lin)$ are equivalent to:
\begin{itemize}
\item[$(\U)$]  the finest uniformity $\U(X)$ of $X$ has a $\mathfrak G$-base.
\end{itemize}
\item
If the space $X$ is closed-$\bar G_\delta$ and strongly open-Reznichenko (which happens if $X$ is first-countable and perfectly normal), then $(\Lc),(\Lin)$  are equivalent to:
\begin{itemize}
\item[$(M\sigma)$] $X$ is metrizable and $\sigma$-compact.
\end{itemize}
\item
If the space $X$ is countable, then the conditions $(\Lc),(\Lin)$ are equivalent to:
\begin{itemize}
\item[$(\mathfrak G)$] the space $X$ has a local $\mathfrak G$-base.
\end{itemize}
\item
If the space $X$ is Ascoli (in particular, a $k$-space), then  $(\Lc),(\Lin)$  are equivalent to:
\begin{itemize}
\item[$(C_k^2)$] the double function space $C_k(C_k(X))$ has a local $\mathfrak G$-base.
\end{itemize}
\item If $\w_1<\mathfrak b$, then the conditions $(\Lc),(\Lin)$ are equivalent to:
\begin{itemize}
\item[$(\w\U)$]  the finest uniformity $\U(X)$ of $X$ is $\w$-narrow and has a $\mathfrak G$-base.
\end{itemize}
\item If $\w_1=\mathfrak b$, then  the conditions $(\Lc),(\Lin)$ are not equivalent to $(\w\U)$.

\end{enumerate}
\end{theorem}

\begin{proof} Identify the topological space $X$ with the uniform space $X$ equipped with the finest uniformity $\U(X)$ on $X$.
 Then $\Lin(X)=\Lin_u(X)$, $\Lc(X)=\Lc_u(X)$, and $C_u(X)=C_\w(X)=C(X)$.
\smallskip

1. The implications $(\U\hat{\mathsf  C})
\Ra(\Lin)\Ra(\Lc)\Ra(\U \mathsf{C})$ follows from Corollary~\ref{c:Lin-dom} and the  reductions
$$\Tau_0(\Lin(X))\succcurlyeq \Tau_0(\Lc(X))\cong \U(X)^\w\times C(X),$$
 established in Theorem~\ref{t:Lin+Lc-reduct}.  The implication $(\U \mathsf C)\Ra(\U  \hat{\mathsf C})$ is trivial. This establishes the equivalences
$(\U\mathsf C)\Leftrightarrow (\U \hat {\mathsf C})\Leftrightarrow(\Lin)\Leftrightarrow(\Lc)$.
\smallskip

2. By Theorems~\ref{t:dominat} and  Proposition~\ref{p:G+dom->K}, the equivalent conditions $(\Lc),(\Lin),(\U\hat{\mathsf C})$ imply that any continuous metrizable image of $X$ is $\sigma$-compact, the function space $C_p(X)$ is $K$-analytic and all finite powers of $X$ are countably tight.
\smallskip

3. If $X$ is separable or $\w_1<\mathfrak b$, then by  Theorem~\ref{t:G-base->sigma} or Corollary~\ref{c:w1<b}, the condition $(\U\hat{\mathsf C})$ implies that $X$ is a countable union of metrizable compact subsets.\smallskip

4. If the space $X$ is separable, then by Corollary~\ref{c:sepun}, $(\U)\Ra(\Lc)$. Combining this implication with the equivalences $(\Lc)\Leftrightarrow(\Lin)\Leftrightarrow(\U\mathsf C)\Leftrightarrow(\U\hat{\mathsf C})\Ra(\U)$, we obtain the desired equivalences $(\Lc)\Leftrightarrow(\Lin)\Leftrightarrow(\U\mathsf C)\Leftrightarrow(\U\hat{\mathsf C})\Leftrightarrow(\U)$.
\smallskip

5. If the space $X$ is closed-$\bar G_\delta$ and strongly open-Reznichenko, then we can apply Theorems~ \ref{t:metr} and \ref{t:dominat} to conclude that $(\U\hat{\mathsf C})\Ra(M\sigma)$. On the other hand, the implication $(M\sigma)\Ra(\U\hat{\mathsf C})$ follows from Proposition~\ref{p:sigma'} and Theorem~\ref{t:t:s->dom}.
\smallskip

6. If the space $X$ is countable, then $(\mathfrak G)\Leftrightarrow(\U)$ by \cite{LPT}. Since $X$ is separable, we have also the equivalences $(\U)\Leftrightarrow (\Lc)\Leftrightarrow(\Lin)$.
\smallskip

7. Assume that the space $X$ is Ascoli. In this case every compact subset of $C_k(X)$ is equicontinuous, which implies that the space $C_k(C_k(X))$ carries the topology of uniform convergence on pointwise bounded equicontinuous sets and hence $\Lc(X)=\Lc_u(X)\subset C_u^*(X)\subset C_k(C_k(X))$. If the space $C_k(C_k(X))$ has a local $\mathfrak G$-base, then so does its subspace $\Lc(X)$, which means that $(C_k^2)\Ra(\Lc)$.

Now assume conversely that the free locally space $\Lc(X)$ has a local $\mathfrak G$-base. By Theorem \ref{t:u} we have the equivalence $\E(C_u(X))\cong \w^{\w}$. The Ascoli property of $X$ implies that $\K(C_k(X))\cong \E(C_u(X))\cong\w^\w$. We obtained that
the function space $C_k(X)$ has a compact resolution swallowing compact sets and hence,
according to a theorem of Ferrando and K{\c{a}}kol \cite{feka},
the space $C_k(C_k(X))$ has a local $\mathfrak G$-base.
\smallskip

8.  Corollary~\ref{c:narrow} yields the implication $(\U\hat{\mathsf C})\Ra(\w\U)$. Now assume that $(\w\U)$ holds.  If $\w_1<\mathfrak b$, then by Theorem~\ref{t:ws}, the space $X$ is separable. Applying Corollary~\ref{c:sepun} we conclude that the space $\Lc(X)$ has a local $\mathfrak G$-base. So, $(\w\U)\Ra(\Lc)\Leftrightarrow(\Lin)$.

9. By Example~\ref{ex:PU}, under $\w_1=\mathfrak b$ there exists a non-discrete Tychonoff $P$-space $X$ satisfying the condition $(\w\U)$. By Remark~\ref{rem:notC}, the set $C(X)$ is not $\w^\w$-dominated in $\IR^X$, which means that $X$ does not satisfy the condition $(\U\hat{\mathsf C})$ and hence does not satisfy the conditions $(\Lc)$, $(\Lin)$, equivalent to $(\U\hat{\mathsf C})$.
\end{proof}

\begin{remark} By \cite{BL2}, there exists a sequential countable $\aleph_0$-space $X$ with a unique non-isolated point which fails to have a local $\mathfrak G$-base. For this space $X$ the spaces $\Lc(X)$ and $\Lin(X)$ do not have local $\mathfrak G$-bases.
\end{remark}

The following conclusion  follows immediately from Theorem \ref{t:Lin+Lc}.
\begin{corollary}\label{cor:metr} For a metrizable space $X$ its free locally convex space
 $\Lc(X)$ has a local $\mathfrak G$-base if and only if its free topological vector space $\Lin(X)$ has a local $\mathfrak G$-base if and only if the space
$X$ is $\sigma$-compact.
\end{corollary}

It is interesting to mention that the equivalence $(\mathfrak G)\Leftrightarrow(\U)$ (holding for countable spaces) can not be  generalized to cosmic $\sigma$-compact spaces.

\begin{example}\label{e:cosmic} There exists a cosmic $\sigma$-compact space $X$ which is first-countable
 (and hence has a local $\mathfrak G$-base) but its finest uniformity $\U(X)$ fails to have a $\mathfrak G$-base.
\end{example}

\begin{proof} By \cite[4.11]{Ban}, there exists  a first-countable comic $\sigma$-compact space $X$, which is not an $\aleph_0$-space and hence is not metrizable. Being cosmic and $\sigma$-compact, the space $X$ can be written as a countable union of metrizable compact subsets. So, it is hereditarily Lindel\"of and hence closed-$\bar G_\delta$. Being first-countable, the space $X$ is strongly open-Reznichenko. Since $X$ is not metrizable, we can apply Theorem~\ref{t:metr} and conclude that the finest uniformity of $X$ fails to have a $\mathfrak G$-base.
\end{proof}

\begin{remark} Since every Tychonoff $k$-space is Ascoli, the equivalence $(\Lc)\Leftrightarrow(C_k^2)$ proved in Theorem~\ref{t:Lin+Lc} answers affirmatively Question 4.18 posed in \cite{GabKakLei_2}.
\end{remark}

\begin{remark} Free topological vector spaces $\Lin(X$) are introduced and studied in a recent preprint \cite{GabMor}, but with completely different approach and results.
Note also that several results about the existence of local $\mathfrak G$-bases in free locally convex spaces  were proved independently (and using different arguments than ours) in \cite{GK_L(X)}.
In particular, the authors of  \cite{GK_L(X)} proved the equivalence $(\Lc)\Leftrightarrow(M\sigma)$ for metrizable spaces $X$ and the equivalence $(\Lc)\Leftrightarrow (\mathfrak G)$ for countable Ascoli spaces $X$.
\end{remark}

Having in mind the equivalence $(\Lc)\Leftrightarrow(\Lin)$ in Theorem~\ref{t:Lin+Lc}, we can ask the following question.

\begin{problem}\label{not_equal}
 Does there exist a uniform space $X$ such that $\Lc_u(X)$ has a local $\mathfrak G$-base,
 but $\Lin_u(X$) fails to have a local $\mathfrak G$-base?
\end{problem}

\section{Uniform quotients and inductive limits}

We say that a surjective map $f:X\to Y$ between uniform spaces is {\em uniformly quotient} if an entourage $U\subset Y\times Y$ belongs to the uniformity $\U(Y)$ if and only if its preimage $(f\times f)^{-1}(U)=\{(x,y)\in X\times X:(f(x),f(y))\in U\}$ belongs to the uniformity $\U(X)$.
This definition implies that each uniformly quotient map between uniform spaces is uniformly continuous.

\begin{lemma}\label{l:quot} For a uniformly quotient map $f:X\to Y$ between uniform spaces the induced linear operators $\Lc f:\Lc_u(X)\to\Lc_u(Y)$ and $\Lin f:\Lin_u(X)\to\Lin_u(Y)$ are open.
\end{lemma}

\begin{proof} Let $Z=(\Lc f)^{-1}(0)$ be the kernel of the operator $\Lc f$, let $\Lc_u(X)/Z$ be the quotient locally convex space, and $q:\Lc_u(X)\to \Lc_u(X)/Z$ be the quotient operator. It follows that $\Lc f=\tilde f\circ q$ for some continuous bijective homomorphism $\tilde f:\Lc_u(X)/Z\to \Lc_u(Y)$. The equality $\tilde f\circ q\circ\delta_X=\Lc f\circ \delta_X=\delta_Y\circ f$ implies the equality $q\circ \delta_X=\tilde f^{-1}\circ\delta_Y\circ f$. Taking into account that the map $f$ is uniformly quotient and the map $q\circ \delta_X=\tilde f^{-1}\circ\delta_Y\circ f$ is uniformly continuous, we conclude that the map $\tilde f^{-1}\circ \delta_Y$ is uniformly continuous. Then the definition of the free locally convex space $\Lc_u(Y)$ implies that the linear operator $\tilde f^{-1}:\Lc_u(Y)\to \Lc_u(X)/Z$ is continuous. So, $\tilde f$ is a topological isomorphism. Since the quotient operator $q:\Lc_u(X)\to \Lc_u(X)/Z$ is open, so is the operator $\Lc f=\tilde f\circ q:\Lc_u(X)\to\Lc_u(Y)$.

By analogy we can prove that the operator $\Lin f:\Lin_u(X)\to\Lin_u(Y)$ is open.
\end{proof}

Since local $\mathfrak G$-bases are preserved by open continuous maps, we get the following corollary.

\begin{proposition}\label{p:quot} Let $f:X\to Y$ be a uniformly quotient map between uniform spaces. If the free (locally convex)  topological vector space $\Lin_u(X)$ (resp. $\Lc_u(X))$) has a local $\mathfrak G$-base, then the space $\Lin_u(Y)$ (resp. $\Lc_u(Y)$) has a local $\mathfrak G$-base, too.
\end{proposition}

A map $f:X\to Y$ between topological spaces is called {\em $\IR$-quotient} if a function $\varphi:Y\to \IR$ is continuous if and only if the composition $\varphi\circ f:X\to\IR$ is continuous. It is clear that each quotient map is $\IR$-quotient.

\begin{corollary}\label{c:Rquot} Let $f:X\to Y$ be an $\IR$-quotient map of Tychonoff spaces. If the free (locally convex) topological vector space $\Lin(X)$ (resp. $\Lc(X))$) has a local $\mathfrak G$-base, then the space $\Lin(Y)$ (resp. $\Lc(Y)$) has a local $\mathfrak G$-base, too.
\end{corollary}

\begin{proof} In view of Proposition \ref{p:quot}, the conclusion follows  as soon as  we check that the $\IR$-quotient map $f:X\to Y$ is uniformly quotient (with respect to the finest uniformities $\U(X)$ and $\U(Y)$). Given a continuous pseudometric $d_Y$ on $Y$ we need to show that the pseudometric $d_X=d_Y(f\times f)$ is continuous. By the triangle inequality, it suffices to show that for every point $x_0\in X$ the map $d_X(x_0,\cdot):X\to \IR$, $d_X(x_0,\cdot):x\mapsto d_X(x_0,x)=d_Y(f(x_0),f(x))$, is continuous. Since the map $f:X\to Y$ is $\IR$-quotient, the continuity of the map $d_Y(f(x_0),\cdot):Y\to\IR$ implies the continuity of the map $d_X(x_0,\cdot)=d_Y(f(x_0),\cdot)\circ f$.
\end{proof}

We say that a topological space $X$ carries the {\em inductive topology with respect to a family}  $(X_n)_{n\in\omega}$ of subsets of $X$
if $V\subseteq X$ is open in $X$ if and only if $V\cap X_n$ is open in $X_n$ for every $n\in\omega$.

\begin{theorem}\label{t:limit} Assume that a Tychonoff space $X$ has the inductive topology with respect to a countable cover $\{ X_n\}_{n\in\omega}$ of $X$.
If for every $n\in\w$ the free (locally convex) topological vector space $\Lin(X_n)$ (resp. $\Lc(X_n)$) has a local $\mathfrak G$-base, then the free (locally convex) topological vector space $\Lin(X)$ (resp. $\Lc(X)$) has a local $\mathfrak G$-base, too.
\end{theorem}

\begin{proof} Let $\Sigma=\bigcup_{n\in\w}\{n\}\times X_n$ be the topological sum of the topological spaces $X_n$, $n\in\w$.  Observe that the free topological vector space $\Lin(\Sigma)$ can be identified with the subspace $$\cbox_{n\in\w}\Lin(X_n)=\{(x_n)_{n\in\w}\in\square_{n\in\w}\Lin(X_n):\exists n\in\w\;\forall m > n\;\;x_m=0\}$$of the box-product $\square_{n\in\w}\Lin(X_n)$.
This fact was proved in a general form in \cite{Morris_functor}
and is reproved in \cite{GabMor}.
It is clear that the box-product $\square_{n\in\w}\Lin(X_n)$ has a local $\mathfrak G$-base if and only if all spaces $\Lin(X_n)$, $n\in\w$, have a local $\mathfrak G$-base.

Since $X$ is the image of $\Sigma$ under the natural quotient map $q\colon \Sigma\to X$, $q\colon (n,x)\mapsto x$, we can apply Corollary~\ref{c:Rquot} and conclude that the free topological vector space $\Lin(X)$ has a local $\mathfrak G$-base if all spaces $\Lin(X_n)$, $n\in\w$, have local $\mathfrak G$-bases.

The case of the free locally convex spaces can be considered analogously.
\end{proof}

Combining Theorems~\ref{t:limit} and \ref{t:Lin+Lc} we get the following corollary which provides many examples of the free (locally convex) topological vector spaces with  local $\mathfrak G$-bases.

\begin{corollary}\label{c:ind-free} Assume that a Tychonoff space $X$ carries the inductive topology with respect to an increasing cover $(X_n)_{n\in\w}$ of its metrizable $\sigma$-compact subspaces. Then the free locally convex space $\Lc(X)$ and the free topological vector space $\Lin(X)$ have  local $\mathfrak G$-bases.
\end{corollary}

\end{document}